\documentclass[oneside,english,reqno]{amsart}
\usepackage{lmodern}
\usepackage[T1]{fontenc}
\usepackage[latin9]{inputenc}
\usepackage{babel}
\usepackage{units}
\usepackage{amsthm}
\usepackage{amssymb}
\usepackage{graphicx}
\PassOptionsToPackage{normalem}{ulem}
\usepackage{ulem}
\usepackage[unicode=true,pdfusetitle,
 bookmarks=true,bookmarksnumbered=false,bookmarksopen=false,
 breaklinks=false,pdfborder={0 0 1},backref=false,colorlinks=false]
 {hyperref}
\usepackage{breakurl}

\makeatletter

\providecommand{\tabularnewline}{\\}
\numberwithin{equation}{section}
\numberwithin{figure}{section}
\theoremstyle{plain}
\newtheorem{thm}{\protect\theoremname}[section]
  \theoremstyle{definition}
  \newtheorem{defn}[thm]{\protect\definitionname}
  \theoremstyle{remark}
  \newtheorem{rem}[thm]{\protect\remarkname}
  \theoremstyle{plain}
  \newtheorem{prop}[thm]{\protect\propositionname}
  \theoremstyle{plain}
  \newtheorem{algorithm}[thm]{\protect\algorithmname}
  \theoremstyle{plain}
  \newtheorem{lem}[thm]{\protect\lemmaname}
  \theoremstyle{definition}
  \newtheorem{example}[thm]{\protect\examplename}

\makeatother

  \providecommand{\algorithmname}{Algorithm}
  \providecommand{\definitionname}{Definition}
  \providecommand{\examplename}{Example}
  \providecommand{\lemmaname}{Lemma}
  \providecommand{\propositionname}{Proposition}
  \providecommand{\remarkname}{Remark}
\providecommand{\theoremname}{Theorem}

\begin{document}
\title[Nonconvex feasibility: Projections, Newton method]{Nonconvex set intersection problems:\\ From projection methods to the Newton method\\ for super-regular sets} 

\begin{abstract}
The problem of finding a point in the intersection of closed sets
can be solved by the method of alternating projections and its variants.
It was shown in earlier papers that for convex sets, the strategy
of using quadratic programming (QP) to project onto the intersection
of supporting halfspaces generated earlier by the projection process
can lead to an algorithm that converges multiple-term superlinearly.
 The main contributions of this paper are to show that this strategy
can be effective for super-regular sets, which are structured nonconvex
sets introduced by Lewis, Luke and Malick. Manifolds should be approximated
by hyperplanes rather than halfspaces. We prove the linear convergence
of this strategy, followed by proving that superlinear and quadratic
convergence can be obtained when the problem is similar to the setting
of the Newton method. We also show an algorithm that converges at
an arbitrarily fast linear rate if halfspaces from older iterations
are used to construct the QP.
\end{abstract}

\author{C.H. Jeffrey Pang}

\curraddr{Department of Mathematics\\ 
National University of Singapore\\ 
Block S17 08-11\\ 
10 Lower Kent Ridge Road\\ 
Singapore 119076 }

\email{matpchj@nus.edu.sg}

\date{\today{}}


\maketitle
\tableofcontents{}

\section{Introduction}

For finitely many closed sets $K_{1},\dots,K_{m}$ in $\mathbb{R}^{n}$,
the \emph{Set Intersection Problem }(SIP) is stated as:
\begin{equation}
\mbox{(SIP):}\quad\mbox{Find }x\in K:=\bigcap_{i=1}^{m}K_{i}\mbox{, where }K\neq\emptyset.\label{eq:SIP}
\end{equation}
 One assumption on the sets $K_{i}$ is that projecting a point in
$\mathbb{R}^{n}$ onto each $K_{i}$ is a relatively easy problem. 

A popular method of solving the SIP is the \emph{Method of Alternating
Projections }(MAP), where one iteratively projects a point through
the sets $K_{i}$ to find a point in $K$. For more on the background
and recent developments of the MAP and its variants, we refer the
reader to \cite{BB96_survey,BR09,EsRa11}, as well as \cite[Chapter 9]{Deustch01}
and \cite[Subsubsection 4.5.4]{BZ05}. We refer to the references
mentioned earlier for a commentary on the applications of the SIP
for the convex case (i.e., when all the sets $K_{i}$ in \eqref{eq:SIP}
are convex)

\subsection{The convex SIP}

One problem of the MAP is slow convergence. As discussed in the previously
mentioned references, in the presence of a regular intersection property,
one can at best expect linear convergence of the MAP. A few acceleration
methods were explored. The papers \cite{GPR67,GK89,BDHP03} explored
the acceleration of the MAP using a line search in the case where
$K_{i}$ are linear subspaces. See also the papers \cite{HernandezRamosEscalanteRaydan2011,Pang_subBAP}
for newer research for this particular setting. 

In \cite{cut_Pang12}, we looked at a different method for the convex
SIP (i.e., the SIP \eqref{eq:SIP} when the sets $K_{i}$ are all
convex). Each projection generates a halfspace containing the intersection
of the sets $K$, and one can project onto the intersection of a number
of these halfspaces using standard methods in quadratic programming
(for example an active set method \cite{Goldfarb_Idnani} or an interior
point method). We call this the SHQP (supporting halfspace and quadratic
programming) strategy. This strategy is illustrated in Figure \ref{fig:alt-proj-compare}.
We refer to \cite{cut_Pang12} for more on the history on the SHQP
strategy, and we point out a few earlier papers that had some ideas
of the SHQP strategy \cite{Pierra84,G-P98,G-P01,BausCombKruk06,Polak_Mayne79,Mayne_Polak_Heunis81}. 

\begin{figure}[h]
\begin{tabular}{|c|c|}
\hline 
\includegraphics[scale=0.3]{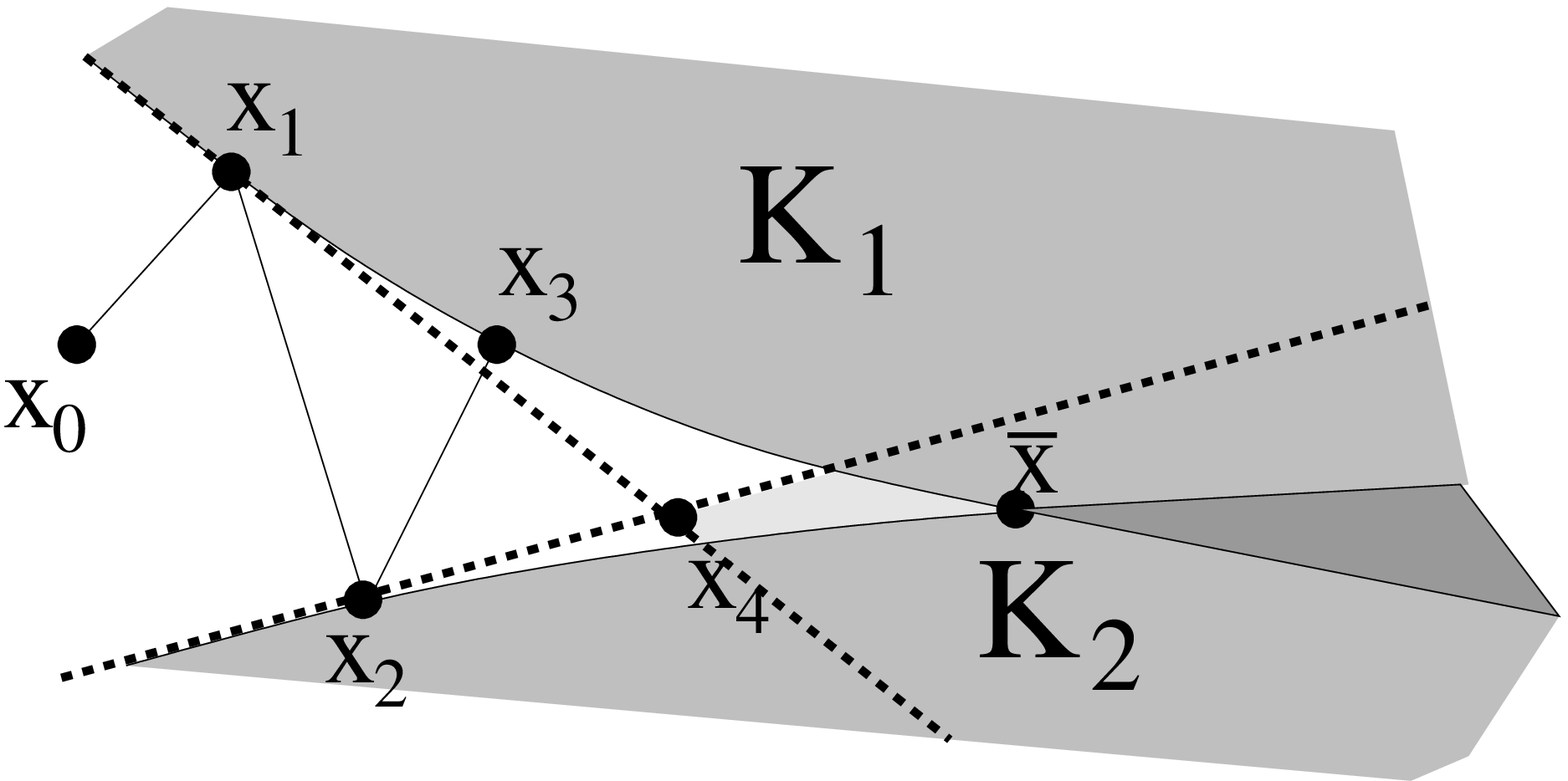} & \includegraphics[scale=0.3]{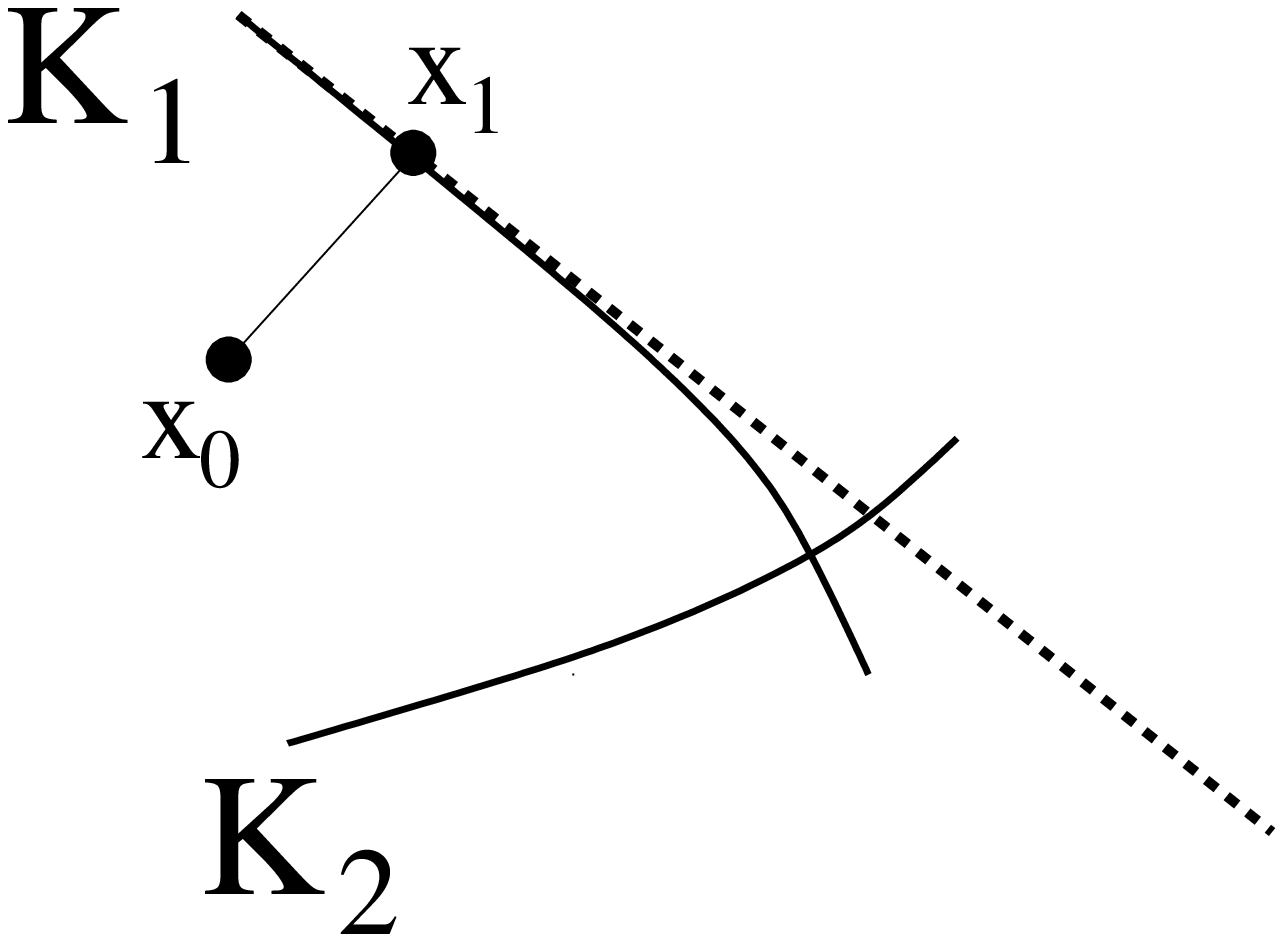}\tabularnewline
\hline 
\end{tabular}

\caption{\label{fig:alt-proj-compare}Refer to the diagram on the left. The
method of alternating projections on two convex sets $K_{1}$ and
$K_{2}$ in $\mathbb{R}^{2}$ with starting iterate $x_{0}$ arrives
at $x_{3}$ in three iterations. The point $x_{4}$ is the projection
of $x_{1}$ onto the intersection of halfspaces generated by projecting
onto $K_{1}$ and $K_{2}$ earlier. One can see that $d(x_{4},K_{1}\cap K_{2})<d(x_{3},K_{1}\cap K_{2})$,
illustrating the potential of the SHQP (supporting halfspace and quadratic
programming) strategy elaborated in \cite{cut_Pang12}. The diagram
on the right shows that such a heuristic need not be effective for
nonconvex sets.}
\end{figure}

The main result in \cite{cut_Pang12} is to show the following: For
a convex SIP satisfying the linearly regular intersection property
(Definition \ref{def:lin-reg-int}), we have an algorithm that achieves
multiple-term superlinear convergence if enough halfspaces generated
from earlier projections are stored to form the quadratic programs
to be solved in later iterations. While the proof of this result suggests
keeping an impractically huge number of halfspaces to guarantee the
fast convergence, simple examples like the one in Figure \ref{fig:alt-proj-compare}
suggests that the number of halfspaces that need to be used to obtain
the fast convergence can actually be quite small.

\subsection{The nonconvex SIP}

We quote from \cite{LLM09_lin_conv_alt_proj} on the applications
and background of the SIP in the nonconvex case (i.e., when the sets
$K_{i}$ in \eqref{eq:SIP} are not known to be convex): An example
of a nonconvex set that is easy to project onto is the set of matrices
with some fixed rank. The method of alternating projections for nonconvex
problems appear in areas such as inverse eigenvalue problems \cite{ChenChu96SINUM,Chu95SIMAX},
pole placement \cite{Orsi06SIMAX,YangOrsi06}, information theory
\cite{TroppDhillonHeathStrohmer05}, low-order control design \cite{GrigoriadisBeran2000SIAM,GrigoriadisSkelton96,OrsiHelmkeMoore06},
and image processing \cite{BauschkeCombettesLuke02,Marchesini_Tu_Wu_2014,WeberAllebach86}.
Previous convergence results on nonconvex alternating projection algorithms
have been uncommon, and have either focused on a very special case
(see, for example \cite{ChenChu96SINUM,LewisMalick08}), or have been
much weaker than for the convex case \cite{CombettesTrussell90,TroppDhillonHeathStrohmer05}.
For more discussion, see \cite{LewisMalick08}. More recent works
on the nonconvex SIP include \cite{BauschkeLukePhanWang13a,BauschkeLukePhanWang13b,HesseLuke12}.
See also \cite{AttouchBolteRedontSoubeyran2010}.

For the nonconvex problem, the projection onto a nonconvex set need
not generate a supporting halfspace. It is easy to construct examples
such that the halfspace generated by the projection process will not
contain any point in the intersection. (See for example the diagram
on the right in Figure \ref{fig:alt-proj-compare}.) The notion of
super-regularity (See Definition \ref{def:super-regular}) was first
defined in \cite{LLM09_lin_conv_alt_proj}. They also showed how super-regularity
is connected to various other well-known properties in variational
analysis. In the presence of super-regularity, they established the
linear convergence of the MAP.

\subsection{\label{sub:Contrib}Contributions of this paper}

The main contribution of this paper is to make two observations about
super-regular sets. The first observation is that once a point is
close enough to a super-regular set, the projection onto this set
produces a halfspace that locally separates a point from the set.
(This observation is used to prove Claim (a) in Theorem \ref{thm:loc-lin-conv}.)
With this observation, the SHQP strategy can be carried over to super-regular
sets. The second observation is that if one of the sets is a manifold,
then we can use a hyperplane to approximate the manifold instead of
using a halfspace in the QP subproblem and still obtain convergence
of our algorithms. See \eqref{eq:manifold-ppty}. 

In Section \ref{sec:local-strat}, we show that under typical conditions
in the study of alternating projections, an algorithm (Algorithm \ref{alg:basic-alg})
that has a sequence of projection steps and SHQP steps that visits
all the sets will converge linearly to a point in the intersection.
In Section \ref{sec:Newton-method}, we show that the SHQP strategy
applied to find a point in the intersection of manifolds and super-regular
sets with only one unit normal on its boundary points will converge
superlinearly. The convergence is quadratic under added conditions.
This makes a connection to the Newton method. Lastly, in Section \ref{sec:fast-alg},
we show that arbitrary fast linear convergence is possible when enough
halfspaces from previous iterations are kept to form the quadratic
programs to accelerate later iterations.

\subsection{Notation}

The notation we use are fairly standard. We let $\mathbb{B}(x,r)$
be the closed ball with center $x$ and radius $r$, and we denote
the projection onto a set $C$ by $P_{C}(\cdot)$.

\section{Preliminaries}

In this section, we recall some definitions in nonsmooth analysis
and some basic background material on the theory of alternating projections
that will be useful for the rest of the paper.
\begin{defn}
\label{def:normal-cones}(Normal cones and Clarke regularity) For
a closed set $C\subset\mathbb{R}^{n}$, the \emph{regular normal cone}
at $\bar{x}$ is defined as 
\begin{equation}
\hat{N}_{C}(\bar{x}):=\{y\mid\langle y,x-\bar{x}\rangle\leq o(\|x-\bar{x}\|)\mbox{ for all }x\in C\}.\label{eq:def-normal-1}
\end{equation}
The \emph{limiting normal cone} at $\bar{x}$ is defined as 
\begin{equation}
N_{C}(\bar{x}):=\{y\mid\mbox{there exists }x_{i}\xrightarrow[C]{}\bar{x},\, y_{i}\in\hat{N}_{C}(x_{i})\mbox{ such that }y_{i}\to y\}.\label{eq:def-normal-2}
\end{equation}
When $\hat{N}_{C}(\bar{x})=N_{C}(\bar{x})$, then $C$ is \emph{Clarke
regular }at $\bar{x}$. If $C$ is Clarke regular at all points, then
we simply say that it is Clarke regular.
\end{defn}
An important tool for our analysis for the rest of the paper is the
following notion of regularity of nonconvex sets.
\begin{defn}
\label{def:super-regular}\cite[Proposition 4.4]{LLM09_lin_conv_alt_proj}(Super-regularity)
A closed set $C\subset\mathbb{R}^{n}$ is \emph{super-regular} \emph{at
a point} $\bar{x}\in C$ if, for all $\delta>0$ we can find a neighborhood
$V$ of $\bar{x}$ such that 
\[
\langle z-y,v\rangle\leq\delta\|z-y\|\|v\|\mbox{ for all }z,y\in C\cap V\mbox{ and }v\in N_{C}(y).
\]
We say that $C$ is super-regular if it is super-regular at all points.
\end{defn}
The discussion in \cite{LLM09_lin_conv_alt_proj} also shows that
\begin{enumerate}
\item \label{enu:Super-reg-Clarke}Super-regularity at a point implies Clarke
regularity there \cite[Corollary 4.5]{LLM09_lin_conv_alt_proj}. (The
converse is not true \cite[Example 4.6]{LLM09_lin_conv_alt_proj}.)
\item Either amenability at a point or prox-regularity at a point implies
super-regularity there \cite[Propositions 4.8 and 4.9]{LLM09_lin_conv_alt_proj}.
\end{enumerate}
We assume that all the sets involved in this paper are super-regular.
In view of property \eqref{enu:Super-reg-Clarke}, we will not need
to distinguish between $\hat{N}_{C}(\bar{x})$ and $N_{C}(\bar{x})$
for the rest of the paper. 
\begin{rem}
(On manifolds) It is clear that if $M$ is a smooth manifold in the
usual sense, then $M$ is super-regular. Moreover, 
\begin{equation}
\mbox{For all }x\in M\mbox{, }v\in N_{M}(x)\mbox{ implies }-v\in N_{M}(x).\label{eq:manifold-ppty}
\end{equation}
For the rest of our discussions, we shall let a manifold be a super-regular
set satisfying \eqref{eq:manifold-ppty}.
\end{rem}
The following property relates $d(x,\cap_{l=1}^{m}K_{l})$ to $\max_{1\leq l\leq m}d(x,K_{l})$.
\begin{defn}
\label{def:Loc-lin-reg}(Local metric inequality) We say that a collection
of closed sets $K_{l}\subset\mathbb{R}^{n}$, $l=1,\dots,m$ satisfies
the \emph{local metric inequality }at $\bar{x}$ if there is a $\beta>0$
and a neighborhood $V$ of $\bar{x}$ such that
\begin{equation}
d(x,\cap_{l=1}^{m}K_{l})\leq\beta\max_{1\leq l\leq m}d(x,K_{l})\mbox{ for all }x\in V.\label{eq:loc-metric-ineq}
\end{equation}

\end{defn}
A concise summary of further studies on the local metric inequality
appears in \cite{Kruger_06}, who in turn referred to \cite{BBL99,Iof00,Ngai_Thera01,NgWang04}
on the topic of local metric inequality and their connection to metric
regularity. Definition \ref{def:Loc-lin-reg} is sufficient for our
purposes. The local metric inequality is useful for proving the linear
convergence of alternating projection algorithms \cite{BB93_Alt_proj,LLM09_lin_conv_alt_proj}.
See \cite{BB96_survey} for a survey. 
\begin{defn}
\label{def:lin-reg-int}(Linearly regular intersection) For closed
sets $K_{l}\subset\mathbb{R}^{n}$, we say that $\{K_{l}\}_{l}$ has
\emph{linearly regular intersection }at $x\in K:=\cap_{l=1}^{m}K_{l}$
if the following condition holds:
\begin{equation}
\mbox{If }\sum_{l=1}^{m}v_{l}=0\mbox{ for some }v_{l}\in N_{K_{l}}(x)\mbox{, then }v_{l}=0\mbox{ for all }l\in\{1,\dots,r\}.\label{eq:CQ-1}
\end{equation}
The linearly regular intersection property appears in \cite[Theorem 6.42]{RW98}
as a condition for proving that $N_{\cap_{l=1}^{m}K_{l}}(x)=\sum_{l=1}^{m}N_{K_{l}}(x)$.
As discussed in \cite{Kruger_06} and related papers, linearly regular
intersection is related to the sensitivity analysis of the SIP \eqref{eq:SIP}.
Linearly regular intersection implies the linear convergence of the
method of alternating projections. Furthermore, linearly regular intersection
implies local metric inequality, but the converse is not true.
\end{defn}
The following easy and well known principle is used to prove the Fej\'{e}r
monotonicity of iterates in Theorems \ref{lem:conv-alg} and \ref{thm:arb-lin-conv}.
\begin{prop}
\label{prop:fejer-principle}(Fej\'{e}r monotonicity) Suppose $C$
is a closed convex set in $\mathbb{R}^{n}$, with $x\notin C$ and
$y\in C$. Then for any $\lambda\in[0,1]$, 
\[
\|y-[P_{C}(x)+\lambda(P_{C}(x)-x)]\|\leq\|y-x\|,
\]
and the inequality is strict if $\lambda\in[0,1)$.
\end{prop}

\section{\label{sec:local-strat}Basic local convergence for super-regular
SIP}

In the absence of additional information on the global structure of
a nonconvex SIP, the analysis of convergence must necessarily be local.
In this section, we discuss how super-regularity can give a halfspace
that locally separates a point from the intersection of the sets.
This leads to the local linear convergence of an alternating projection
algorithm that incorporates QP steps whenever possible. 

We begin with the algorithm that we study for this section.
\begin{algorithm}
\label{alg:basic-alg}(Basic algorithm) Let $K_{l}$ be (not necessarily
convex) closed sets in $\mathbb{R}^{n}$ for $l\in\{1,\dots,m\}$.
From a starting point $x_{0}\in\mathbb{R}^{n}$, this algorithm finds
a point in the intersection $K:=\cap_{l=1}^{m}K_{l}$. \\
$\,$\\
01 For iteration $i=0,1,\dots$\\
02 $\quad$Set $x_{i}^{0}=x_{i}$.\\
03 $\quad$Find sets $S_{1}$, $\dots$, $S_{m}\subset\{1,\dots,m\}$
such that $\cup_{i=1}^{m}S_{i}=\{1,\dots,m\}$.\\
04 $\quad$For $j=1,\dots,m$\\
05 $\quad\quad$Find $x_{i,j,l}\in P_{K_{l}}(x_{i}^{j-1})$ for all
$l\in S_{j}$\\
06 $\quad\quad$For $l\in S_{j}$, define halfspace/ hyperplane $H_{i,j,l}$
by 
\[
H_{i,j,l}:=\begin{cases}
\{x:\langle x_{i}^{j-1}-x_{i,j,l},x-x_{i,j,l}\rangle=0\} & \mbox{ if }K_{l}\mbox{ is a manifold}\\
\{x:\langle x_{i}^{j-1}-x_{i,j,l},x-x_{i,j,l}\rangle\leq0\} & \mbox{ otherwise}.
\end{cases}
\]
07 $\quad\quad$Define the polyhedron $F_{i}^{j}$ by $F_{i}^{j}=\cap_{(k,l)\in\tilde{S}_{i}^{j}}H_{i,k,l}$,
where\\
08 $\quad\quad$$\tilde{S}_{i}^{j}\subset\{1,\dots,m\}\times\{1,\dots,m\}$
is such that $\{j\}\times S_{j}\subset\tilde{S}_{i}^{j}$ and\begin{subequations}\label{eq:def-S-i-j}
\begin{eqnarray}
\tilde{S}_{i}^{j}: & = & \big\{(k,l):l\in S_{k},k\in\{1,\dots,j\},\mbox{ and }\label{eq:def-S-i-j-1}\\
 &  & \phantom{\big\{(k,l):}(k_{1},l),(k_{2},l)\in\tilde{S}_{i}^{j}\mbox{ implies }k_{1}=k_{2}\big\}.\label{eq:def-S-i-j-2}
\end{eqnarray}
\end{subequations}09 $\quad$$\quad$Set $x_{i}^{j}=P_{F_{i}^{j}}(x_{i}^{j-1})$.\\
10 $\quad$end for\\
11 $\quad$Set $x_{i+1}=x_{i}^{m}$.\\
12 end
\end{algorithm}
We allow some of the $S_{j}$'s to be empty as long as the condition
$\cup_{i=1}^{m}S_{i}=\{1,\dots,m\}$ is satisfied. When $S_{j}=\{j\}$
and $\tilde{S}_{i}^{j}=\{(j,j)\}$ for all $i,j$, Algorithm \ref{alg:basic-alg}
reduces to the alternating projection algorithm. Algorithm \ref{alg:basic-alg}
has the given design because we believe that by performing QP steps
with polyhedra that bound the sets $K_{l}$ better, the convergence
to a point in $K$ can be accelerated. Yet, we still retain the flexibility
of the size of the QPs so that each step can be performed with a reasonable
amount of effort.
\begin{rem}
\label{rem:mass-proj}(Mass projection) Another particular case of
Algorithm \ref{alg:basic-alg} we will study in Section \ref{sec:Newton-method}
is when $S_{1}=\{1,\dots,m\}$, $S_{j}=\emptyset$ for all $j\in\{2,\dots,m\}$,
and $\tilde{S}_{i}^{j}=\{j\}\times S_{j}$ for all $i,j\in\{1,\dots,m\}$.
In such a case, Algorithm \ref{alg:basic-alg} is simplified to 
\begin{eqnarray*}
x_{i,1,l} & \in & P_{K_{l}}(x_{i})\\
H_{i,1,l} & = & \begin{cases}
\{x:\langle x_{i}-x_{i,1,l},x-x_{i,1,l}\rangle=0\} & \mbox{ if }K_{l}\mbox{ is a manifold}\\
\{x:\langle x_{i}-x_{i,1,l},x-x_{i,1,l}\rangle\leq0\} & \mbox{ otherwise}
\end{cases}\\
x_{i+1} & = & P_{\cap_{l=1}^{m}H_{i,1,l}}(x_{i}).
\end{eqnarray*}

\begin{rem}
\label{rem:sometimes-empty-intersect}(On the polyhedron $F_{i}^{j}$)
The polyhedron $F_{i}^{j}$ is defined by intersecting some of the
halfspaces/ hyperplanes $H_{i,k,l}$. The line \eqref{eq:def-S-i-j-2}
in \eqref{eq:def-S-i-j} defining $\tilde{S}_{i}^{j}$ ensures that
no two of the halfspaces/ hyperplanes $H_{i,k,l}$ that are intersected
to form $F_{i}^{j}$ come from projecting onto the same set. To see
why we need \eqref{eq:def-S-i-j-2}, observe that we can draw two
tangent lines to a manifold in $\mathbb{R}^{2}$ that do not intersect,
which would lead to $F_{i}^{j}=\emptyset$. 
\begin{rem}
(Treatment of manifolds) Another feature of this algorithm is that
when $K_{l}$ is a manifold, the set $H_{i,j,l}$ is a hyperplane
instead. Manifolds are super-regular sets. We take advantage of property
\eqref{eq:manifold-ppty} of manifolds to create a more logical algorithm.
The hyperplane is a better approximate of a manifold than a halfspace,
and we may expect faster convergence to a point in $K$ when we use
hyperplanes instead. Another advantage of using hyperplanes is that
quadratic programming algorithms resolve equality constraints (which
are always tight) better than they resolve inequality constraints
(where determining whether each constraint is tight at the optimal
solution requires some effort).
\end{rem}
\end{rem}
\end{rem}
The lemma below will be useful in studying the convergence of the
algorithms throughout this paper.
\begin{lem}
\label{lem:lin-conv-backbone}(Linear convergence conditions) Let
$K$ be a set in $\mathbb{R}^{n}$. Suppose an algorithm generates
iterates $\{x_{i}\}$ such that 
\begin{enumerate}
\item There exists some $\rho\in(0,1)$ such that $d(x_{i+1},K)\leq\rho d(x_{i},K)$,
and
\item there exists a constant $c>0$ such that $\|x_{i+1}-x_{i}\|\leq cd(x_{i},K)$.
\end{enumerate}
Then the sequence $\{x_{i}\}$ converges to a point $\bar{x}\in K$,
and we have, for all $i\geq0$, 
\begin{itemize}
\item [(a)]$\|x_{i}-\bar{x}\|\leq\frac{c}{1-\rho}d(x_{i},K)\leq\frac{c\rho^{i}}{1-\rho}d(x_{0},K)$,
and
\item [(b)]$\mathbb{B}(x_{i+1},\frac{c}{1-\rho}d(x_{i+1},K))\subset\mathbb{B}(x_{i},\frac{c}{1-\rho}d(x_{i},K))$.
\end{itemize}
\end{lem}
\begin{proof}
For any $j\geq0$, we have 
\[
\|x_{i+j+1}-x_{i+j}\|\leq cd(x_{i+j},K)\leq c\rho^{j}d(x_{i},K).
\]
Standard arguments in analysis shows that $\{x_{i}\}$ is a Cauchy
sequence which converges to a point $\bar{x}\in K$. Both parts (a)
and (b) are straightforward. 
\end{proof}
The next result shows how such derived halfspaces relate to the original
halfspaces.
\begin{lem}
\label{lem:derived-halfspaces}(Derived supporting halfspaces) Let
$\bar{x}\in\mathbb{R}^{n}$, and suppose $H_{1}$, $H_{2}$, $\dots$,
$H_{k}$ are $k$ halfspaces containing $\bar{x}$ such that $d(\bar{x},\partial H_{i})$,
the distance from $\bar{x}$ to the boundary of each halfspace $H_{i}$,
is at most $\alpha$. Suppose the normal vectors of each halfspace
$H_{i}$ is $v_{i}$, where $\|v_{i}\|=1$, and the constant $\eta$
defined by 
\begin{equation}
\eta:=\min\left\{ \left\Vert \sum_{i=1}^{k}\lambda_{i}v_{i}\right\Vert :\sum_{i=1}^{k}\lambda_{i}=1,\,\lambda_{i}\geq0\mbox{ for all }i\in\{1,\dots,k\}\right\} \label{eq:eta-defn}
\end{equation}
is positive. (i.e., $\eta\neq0$.) Let $F$ be the intersection of
these halfspaces. Let $\tilde{H}$ be the halfspace containing $F$
produced by projecting from a point $x^{\prime}\notin F$ onto $F$.
In other words, the halfspace $\tilde{H}$ is defined by 
\[
\{x:\langle x^{\prime}-P_{F}(x^{\prime}),x-P_{F}(x^{\prime})\rangle\leq0\}.
\]
Then the distance of $\bar{x}$ from the boundary of $\tilde{H}$
is at most $\frac{1}{\eta}\alpha$.

As a consequence, suppose $H_{i}$ are defined by $H_{i}=\{x:\langle v_{i},x\rangle\leq\alpha\}$.
Let $v=\frac{\sum_{i=1}^{k}\lambda_{i}v_{i}}{\|\sum_{i=1}^{k}\lambda_{i}v_{i}\|}$
for some nonzero vector $\lambda\in\mathbb{R}^{k}$ that has nonnegative
components, and $H$ be $H=\{x:\langle v,x\rangle\leq\frac{\alpha}{\eta}\}$.
Then we have $\cap_{i=1}^{k}H_{i}\subset\tilde{H}\subset H$. \end{lem}
\begin{proof}
We remark that $\eta$ is the distance of the origin to the convex
hull of $\{v_{i}\}_{i=1}^{k}$. We can eliminate halfspaces if necessary
and assume that $k\geq1$, and that $P_{F}(x^{\prime})$ lies on the
boundaries of all the halfspaces. The KKT condition tells us that
$x^{\prime}-P_{F}(x^{\prime})$ lies in the conical hull of $\{v_{i}\}_{i=1}^{k}$.
By Caratheodory's theorem, we can assume that $k$ is not more than
the dimension $n$. We can also eliminate halfspaces if necessary
so that the vectors $\{v_{i}\}_{i=1}^{k}$ are linearly independent. 

Suppose each halfspace $H_{i}$ is defined by $\{x:\langle v_{i},x\rangle\leq b_{i}\}$,
where $b_{i}\in\mathbb{R}$. Since $P_{F}(x^{\prime})$ lies on the
boundaries of the halfspaces $H_{i}$, we have 
\begin{equation}
\langle v_{i},P_{F}(x^{\prime})\rangle=b_{i}\mbox{ for all }i.\label{eq:inn-pdt-eq-b-i}
\end{equation}
Define the hyperslab $S_{i}$ by 
\begin{equation}
S_{i}:=\{x:\langle v_{i},x\rangle\in[b_{i}-\alpha,b_{i}]\}.\label{eq:S-i-hyperslab}
\end{equation}
Since the distance from $\bar{x}$ to the boundaries of each halfspace
$H_{i}$ were assumed to be at most $\alpha$, the point $\bar{x}$
is inside all the hyperslabs $S_{i}$. 

Let $\bar{v}$ be the vector $\frac{x^{\prime}-P_{F}(x^{\prime})}{\|x^{\prime}-P_{F}(x^{\prime})\|}$.
We now study the problem 
\begin{eqnarray}
 & \min_{x} & \left\langle \bar{v},x\right\rangle \label{eq:min-slab-pblm}\\
 & \mbox{s.t.} & x\in S_{i}\mbox{ for all }i\in\{1,\dots,k\}.\nonumber 
\end{eqnarray}
If the above problem were a maximization problem instead, then an
optimizer is $P_{F}(x^{\prime})$. Consider the point $P_{F}(x^{\prime})-\alpha d$,
where $d$ is the direction defined through 
\begin{equation}
\langle v_{i},d\rangle=1\mbox{ for all }i\mbox{, and }d\in\mbox{span}(\{v_{i}\}_{i=1}^{k}).\label{eq:defn-dirn-d}
\end{equation}
Since the vectors $\{v_{i}\}_{i=1}^{k}$ are linearly independent,
such a $d$ exists, and can be calculated by $d=QR^{-T}\mathbf{1}$,
where $\mathbf{1}$ is the vector of all ones, $QR$ is the QR factorization
of $V$, and $V$ is the matrix formed by concatenating the vectors
$\{v_{i}\}_{i=1}^{k}$. We can use \eqref{eq:inn-pdt-eq-b-i} and
\eqref{eq:defn-dirn-d} to calculate that 
\[
\langle v_{i},P_{F}(x^{\prime})-\alpha d\rangle=b_{i}-\alpha\mbox{ for all }i,
\]
so $P_{F}(x^{\prime})-\alpha d$ is on the other boundary of all the
hyperslabs $S_{i}$. Furthermore, since $N_{\cap_{i=1}^{k}S_{i}}(P_{F}(x^{\prime})-\alpha d)=-N_{\cap_{i=1}^{k}S_{i}}(P_{F}(x^{\prime}))$,
we have $-\bar{v}\in N_{\cap_{i=1}^{k}S_{i}}(P_{F}(x^{\prime})-\alpha d)$.
Hence $P_{F}(x^{\prime})-\alpha d$ is a minimizer of \eqref{eq:min-slab-pblm}.

We proceed to find the optimal value of \eqref{eq:min-slab-pblm}.
Since $\bar{v}$ lies in the conical hull of $\{v_{i}\}_{i=1}^{k}$,
$\bar{v}$ can be written as $\frac{V\lambda}{\|V\lambda\|}$, where
$\lambda\in\mathbb{R}_{+}^{k}$ is a vector with nonnegative elements
such that its elements sum to one. We can calculate 
\begin{eqnarray*}
\left(\frac{V\lambda}{\|V\lambda\|}\right)^{T}d & = & \frac{1}{\|V\lambda\|}\lambda^{T}V^{T}QR^{-T}\mathbf{1}\\
 & = & \frac{1}{\|V\lambda\|}\lambda^{T}R^{T}Q^{T}QR^{-T}\mathbf{1}=\frac{1}{\|V\lambda\|}\lambda^{T}\mathbf{1}=\frac{1}{\|V\lambda\|}.
\end{eqnarray*}
By the definition of $\eta$, we have $\frac{1}{\|V\lambda\|}\geq\frac{1}{\eta}$.
This means that the minimum value of \eqref{eq:min-slab-pblm} is
at least $\left\langle \bar{v},P_{F}(x^{\prime})-\alpha d\right\rangle =\left\langle \bar{v},P_{F}(x^{\prime})\right\rangle -\frac{1}{\eta}\alpha$.
Since $\bar{x}\in S_{i}$ for all $i\in\{1,\dots,k\}$, we can deduce
that $\bar{x}$ lies in the hyperslab 
\[
\{x:\langle\bar{v},x\rangle\in[\langle\bar{v},P_{F}(x^{\prime})\rangle-\nicefrac{\alpha}{\eta},\langle\bar{v},P_{F}(x^{\prime})\rangle]\}.
\]
In other words, $\bar{x}$ lies in the halfspace $\{x:\langle\bar{v},x\rangle\leq\langle\bar{v},P_{F}(x^{\prime})\rangle\}$,
and the distance from $\bar{x}$ to the boundary of this halfspace
is at most $\frac{1}{\eta}\alpha$, which is the conclusion we seek.

The final paragraph is easily deduced from the main result.\end{proof}
\begin{rem}
(The formula $\eta$) We remark that the use of the notation $\eta$
in Lemma \ref{lem:derived-halfspaces} is consistent with the notation
of \cite{Kruger_06} and related papers, where the relationship of
the constants related to the sensitivity analysis of the SIP \eqref{eq:SIP}
and linearly regular intersection are studied.
\end{rem}
We now prove our result on the convergence of Algorithm \ref{alg:basic-alg}.
\begin{thm}
\label{thm:loc-lin-conv}(Local linear convergence of general Algorithm)
Suppose $K_{l}$, where $l\in\{1,\dots,m\}$, are super-regular at
$x^{*}\in K=\cap_{l=1}^{m}K_{l}$. Suppose that $\eta$ defined by
\[
\eta:=\min\left\{ \left\Vert \sum_{i=1}^{m}v_{i}\right\Vert :v_{i}\in N_{K_{i}}(x^{*}),\, x^{*}\in K_{i},\,\sum_{i=1}^{m}\|v_{i}\|=1\right\} 
\]
 is positive. (i.e., $\eta\neq0$.) This is equivalent to $\{K_{l}\}_{l=1}^{m}$
having linear regular intersection at $x^{*}$, which in turn implies
that the local metric inequality holds at $x^{*}$. If $x_{0}$ is
sufficiently close to $x^{*}$, then Algorithm \ref{alg:basic-alg}
converges to a point in $K$ Q-linearly (i.e., at a rate bounded above
by a geometric sequence).\end{thm}
\begin{proof}
Since the local metric inequality holds at $x^{*}$, let $\beta\geq1$
and $V$ be a neighborhood of $x^{*}$ such that 
\[
d(x,K)\leq\beta\max_{l}d(x,K_{l})\mbox{ for all }x\in V.
\]
Let \begin{subequations}
\begin{eqnarray}
\rho & = & \sqrt{1+\frac{1}{\beta^{2}m^{3}}+\frac{1}{4\beta^{4}m^{6}}-\frac{1}{\beta^{2}m^{2}}+\frac{1}{2\beta^{3}m^{5}}-\frac{1}{16\beta^{4}m^{8}}+\frac{1}{16\beta^{4}m^{6}}},\label{eq:linear-rho}\\
\mbox{ and }c & = & \sqrt{m}\sqrt{\left[1+\frac{1}{4m^{3}\beta^{2}}\right]^{2}+\frac{1}{16m^{6}\beta^{4}}}\label{eq:linear-c}
\end{eqnarray}
\end{subequations}It is clear to see that if $m\geq2$, then $\rho<1$.
Choose $\delta>0$ such that $\delta\leq\frac{(1-\rho)\eta}{16m^{4}\beta^{2}c}$.
Since $x^{*}$ is super-regular at all sets $K_{l}$, where $l\in\{1,\dots,m\}$,
we can shrink the neighborhood $V$ if necessary so that for all $l\in\{1,\dots,m\}$,
we have 
\[
\langle v,z-y\rangle\leq\delta\|v\|\|z-y\|\mbox{ for all }z,y\in K_{l}\cap V\mbox{ and }v\in N_{K_{l}}(y).
\]
By the outer semicontinuity of the normal cones, we can shrink $V$
if necessary so that for all $x\in V$, we have 
\[
\min\left\{ \left\Vert \sum_{i=1}^{m}v_{i}\right\Vert :v_{i}\in N_{K_{i}}(x),\, x\in K_{i},\,\sum_{i=1}^{m}\|v_{i}\|=1\right\} \geq\frac{\eta}{2}.
\]

Suppose $x_{0}$ is close enough to $x^{*}$ such that $\mathbb{B}(x_{0},\frac{c}{1-\rho}d(x_{0},K))\subset V$.
Provided that we prove conditions (1) and (2) in Lemma \ref{lem:lin-conv-backbone},
we have the convergence of the iterates $\{x_{i}\}$ to some point
$\bar{x}\in K$. The convergence of $\{x_{i}\}$ to $\bar{x}$ would
be at the rate suggested in Lemma \ref{lem:lin-conv-backbone}(a).

If $x\in K\cap\mathbb{B}(x_{i},\frac{c}{1-\rho}d(x_{i},K))$ and $x_{i}^{j-1},x_{i,j,l}\in\mathbb{B}(x_{i},\frac{c}{1-\rho}d(x_{i},K))$,
then 
\begin{eqnarray}
\left\langle \frac{x_{i}^{j-1}-x_{i,j,l}}{\|x_{i}^{j-1}-x_{i,j,l}\|},x-x_{i,j,l}\right\rangle  & \leq & \delta\|x-x_{i,j,l}\|\label{eq:halfspace-QP-alg}\\
 & \leq & \delta\frac{2c}{1-\rho}d(x_{i},K)\nonumber \\
 & \leq & \frac{\eta}{8m^{4}\beta^{2}}d(x_{i},K).\nonumber 
\end{eqnarray}
Define the halfspace $H_{i,j,l}^{+}$ by 
\[
H_{i,j,l}^{+}:=\left\{ x:\left\langle \frac{x_{i}^{j-1}-x_{i,j,l}}{\|x_{i}^{j-1}-x_{i,j,l}\|},x-x_{i,j,l}\right\rangle \leq\frac{\eta}{8m^{4}\beta^{2}}d(x_{i},K)\right\} .
\]
(Note that the halfspace $H_{i,j,l}$ defined in Algorithm \ref{alg:basic-alg}
is similar to $H_{i,j,l}^{+}$ with the exception that the right hand
side of the inequality is zero.) We have $K_{l}\cap\mathbb{B}(x_{i},\frac{c}{1-\rho}d(x_{i},K))\subset H_{i,j,l}^{+}$.
Note that $x_{i}^{j}$ is the projection of $x_{i}^{j-1}$ onto $F_{i}^{J}$.
Define the halfspace $H_{i,j}^{+}$ by 
\begin{equation}
H_{i,j}^{+}:=\left\{ x:\left\langle \frac{x_{i}^{j-1}-x_{i}^{j}}{\|x_{i}^{j-1}-x_{i}^{j}\|},x-x_{i}^{j}\right\rangle \leq\frac{1}{4m^{4}\beta^{2}}d(x_{i},K)\right\} .\label{eq:agg-halfspace-defn}
\end{equation}
By Lemma \ref{lem:derived-halfspaces}, we have 
\begin{equation}
K\cap\mathbb{B}(x_{i},\frac{c}{1-\rho}d(x_{i},K))\subset\cap_{(k,l)\in\tilde{S}_{i}^{j}}H_{i,k,l}^{+}\subset H_{i,j}^{+}.\label{eq:agg-halfspace-works}
\end{equation}
Note that almost exactly the same arguments works if the set $K_{l}$
is a manifold, but we may have to take $-\frac{x_{i}^{j-1}-x_{i}^{j}}{\|x_{i}^{j-1}-x_{i}^{j}\|}$
as the normal vector of $H_{i,j,l}^{+}$ instead and define $H_{i,j}^{+}$
differently, depending on the multipliers in the KKT condition.

\textbf{Claim:} 

(a) If $\|x_{i}^{j}-x_{i}^{j-1}\|\geq\frac{1}{2m^{4}\beta^{2}}d(x_{i},K)$,
then 

$\qquad\qquad d(x_{i}^{j},K)^{2}\leq d(x_{i}^{j-1},K)^{2}-\left[\|x_{i}^{j}-x_{i}^{j-1}\|-\frac{1}{4m^{4}\beta^{2}}d(x_{i},K)\right]^{2}$

$\qquad\qquad\phantom{d(x_{i}^{j},K)^{2}\leq}+\left[\frac{1}{4m^{4}\beta^{2}}d(x_{i},K)\right]^{2}.$

(b) If $\|x_{i}^{j}-x_{i}^{j-1}\|\leq\frac{1}{2m^{4}\beta^{2}}d(x_{i},K)$,
then $d(x_{i}^{j},K)\leq d(x_{i}^{j-1},K)+\frac{1}{2m^{4}\beta^{2}}d(x_{i},K)$. 

Part (b) is obvious. We now prove part (a). Let $y$ be any point
in $P_{K}(x_{i}^{j-1})$, and let $z=P_{H_{i,j}^{+}}(x_{i}^{j-1})$.
See Figure \ref{fig:apply-cosine-rule}, where $d_{2}=\frac{1}{4m^{4}\beta^{2}}d(x_{i},K)$
in view of \eqref{eq:agg-halfspace-defn}. Noting that $\angle zyx_{i}^{j}\geq\pi/2$,
we apply cosine rule to get 
\begin{eqnarray*}
 &  & d(x_{i}^{j},K)^{2}\\
 & \leq & \|y-x_{i}^{j}\|^{2}\\
 & \leq & \|y-z\|^{2}+\|z-x_{i}^{j}\|^{2}\\
 & \leq & \|y-x_{i}^{j-1}\|^{2}-\|x_{i}^{j}-z\|^{2}+\|z-x_{i}^{j}\|^{2}\\
 & = & d(x_{i}^{j-1},K)^{2}-\left[\|x_{i}^{j}-x_{i}^{j-1}\|-\frac{1}{4m^{4}\beta^{2}}d(x_{i},K)\right]^{2}+\left[\frac{1}{4m^{4}\beta^{2}}d(x_{i},K)\right]^{2}.
\end{eqnarray*}
This completes the proof of the claim.

\begin{figure}[!h]
\includegraphics[scale=0.3]{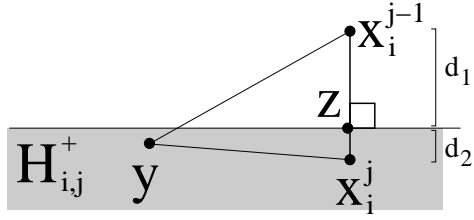}

\caption{\label{fig:apply-cosine-rule}This figure illustrates the proof in
the claim of Theorem \ref{thm:loc-lin-conv}. Note that $d_{1}=\|x_{i}^{j}-x_{i}^{j-1}\|-\frac{1}{4m^{4}\beta^{2}}d(x_{i},K)$
and $d_{2}=\frac{1}{4m^{4}\beta^{2}}d(x_{i},K)$.}

\end{figure}

It now remains the prove conditions (1) and (2) of Lemma \ref{lem:lin-conv-backbone}.
By local metric inequality, there is some $j\in\{1,\dots,m\}$ such
that $d(x_{i},K_{j})\geq\frac{1}{\beta}d(x_{i},K)$. Hence there is
a distance $\|x_{i}^{j}-x_{i}^{j-1}\|$ that will be at least $\frac{1}{m\beta}d(x_{i},K)$.
Making use of the claim earlier, we have the following estimate of
$d(x_{i+1},K)$. 
\begin{eqnarray}
 &  & d(x_{i+1},K)^{2}\label{eq:first-ineq-block}\\
 & \leq & \left[d(x_{i},K)+\frac{m}{2m^{4}\beta^{2}}d(x_{i},K)\right]^{2}-\left[\frac{1}{m\beta}d(x_{i},K)-\frac{1}{4m^{4}\beta^{2}}d(x_{i},K)\right]^{2}\nonumber \\
 &  & +\left[\frac{m}{4m^{4}\beta^{2}}d(x_{i},K)\right]^{2}\nonumber \\
 & = & \left[1+\frac{1}{\beta^{2}m^{3}}+\frac{1}{4\beta^{4}m^{6}}-\frac{1}{\beta^{2}m^{2}}+\frac{1}{2\beta^{3}m^{5}}-\frac{1}{16\beta^{4}m^{8}}+\frac{1}{16\beta^{4}m^{6}}\right]d(x_{i},K)^{2}\nonumber \\
 & = & \rho^{2}d(x_{i},K)^{2}.\nonumber 
\end{eqnarray}
This proves that $d(x_{i+1},K)\leq\rho d(x_{i},K)$. Next, 
\begin{eqnarray}
 &  & \|x_{i+1}-x_{i}\|\label{eq:2nd-ineq-block}\\
 & \leq & \sum_{j=1}^{m}\|x_{i}^{j}-x_{i}^{j-1}\|\nonumber \\
 & \leq & \frac{1}{4m^{3}\beta^{2}}d(x_{i},K)+\sum_{j=1}^{m}\max\left\{ \|x_{i}^{j}-x_{i}^{j-1}\|-\frac{1}{4m^{4}\beta^{2}}d(x_{i},K),0\right\} \nonumber \\
 & \leq & \frac{1}{4m^{3}\beta^{2}}d(x_{i},K)+\underbrace{\sqrt{m\sum_{j=1}^{m}\max\left\{ \|x_{i}^{j}-x_{i}^{j-1}\|-\frac{1}{4m^{4}\beta^{2}}d(x_{i},K),0\right\} ^{2}}}_{(*)}.\nonumber 
\end{eqnarray}
By the analysis in \eqref{eq:first-ineq-block}, the fact that $d(x_{i+1},K)^{2}\geq0$
gives 
\begin{eqnarray*}
0 & \leq & d(x_{i+1},K)^{2}\\
 & \leq & \left[d(x_{i},K)+m\frac{1}{2m^{4}\beta^{2}}d(x_{i},K)\right]^{2}\\
 &  & +\left[\frac{m}{4m^{4}\beta^{2}}d(x_{i},K)\right]^{2}-\sum_{j=1}^{m}\max\left\{ \|x_{i}^{j}-x_{i}^{j-1}\|-\frac{1}{4m^{4}\beta^{2}}d(x_{i},K),0\right\} ^{2}.
\end{eqnarray*}
We thus deduce that the term marked $(*)$ in \eqref{eq:2nd-ineq-block}
is at most 
\[
\sqrt{m}\sqrt{\left[1+\frac{1}{2m^{3}\beta^{2}}\right]^{2}+\frac{1}{16m^{6}\beta^{4}}}d(x_{i},K).
\]
Thus the constant $c$ in Lemma \ref{lem:lin-conv-backbone} can be
taken to be what was given in \eqref{eq:linear-c}. \end{proof}
\begin{rem}
(On the condition $\eta>0$ in Theorem \ref{thm:loc-lin-conv}) The
condition $\eta>0$ is required in the proof of Theorem \ref{thm:loc-lin-conv}
only when $|S_{l}|>1$, when halfspaces are aggregated. So in the
case of alternating projections, the weaker condition of local metric
inequality is sufficient.
\end{rem}

\section{\label{sec:Newton-method}Connections with the Newton method}

To find a point in $\{x\in\mathbb{R}^{n}:F(x)=0\}$ for some smooth
$F:\mathbb{R}^{n}\to\mathbb{R}^{m}$, the method of choice is to use
the Newton method provided that the linear system in the Newton method
can be solved quickly enough. Note that the set $\{x:F(x)=0\}$ can
be written as the intersection of the manifolds $M_{j}:=\{x:F_{j}(x)=0\}$
for $j\in\{1,\dots,m\}$, where $F_{j}:\mathbb{R}^{n}\to\mathbb{R}$
is the $j$th component of $F(\cdot)$. Note that the manifolds $M_{j}$
are of codimension 1. This section gives conditions for which the
SHQP strategy can converge superlinearly or quadratically when the
sets involved satisfy the conditions for fast convergence in the Newton
method. 

The following result was proved in \cite{cut_Pang12} for convex sets,
but is readily generalized to Clarke regular sets, which we do so
now.
\begin{thm}
\label{thm:radiality}(Supporting hyperplane near a point) Suppose
$C\subset\mathbb{R}^{n}$ is Clarke regular, and let $\bar{x}\in C$.
Then for any $\epsilon>0$, there is a $\delta>0$ such that for any
point $x\in[\mathbb{B}_{\delta}(\bar{x})\cap C]\backslash\{\bar{x}\}$
and supporting hyperplane $A$ of $C$ with unit normal $v\in N_{C}(x)$
at the point $x$, we have 
\begin{equation}
\left\langle v,x-\bar{x}\right\rangle \leq\epsilon\|\bar{x}-x\|.\label{eq:little-SOSH}
\end{equation}
\end{thm}
\begin{proof}
Let $\delta$ be small enough so that for any $x\in[\mathbb{B}_{\delta}(\bar{x})\cap C]\backslash\{\bar{x}\}$
and unit normal $v\in N_{C}(x)$, we can find $\bar{v}\in N_{C}(\bar{x})$
such that $\|v-\bar{v}\|<\frac{\epsilon}{2}$ and that $\langle\bar{v},x-\bar{x}\rangle\leq\frac{\epsilon}{2}\|x-\bar{x}\|$.
Then we have 
\begin{eqnarray*}
\langle v,x-\bar{x}\rangle & = & \langle v-\bar{v},x-\bar{x}\rangle+\langle\bar{v},x-\bar{x}\rangle\\
 & \leq & \|v-\bar{v}\|\|x-\bar{x}\|+\frac{\epsilon}{2}\|x-\bar{x}\|\\
 & \leq & \epsilon\|x-\bar{x}\|.
\end{eqnarray*}
Thus we are done.
\end{proof}
We identify a property that will give multiple-term quadratic convergence.
Compare this property to that in Theorem \ref{thm:radiality}.
\begin{defn}
\label{def:SOSH}(Second order supporting hyperplane property) Suppose
$C\subset\mathbb{R}^{n}$ is a closed convex set, and let $\bar{x}\in C$.
We say that $C$ has the \emph{second order supporting hyperplane
(SOSH) property at $\bar{x}$ }(or more simply, $C$ is SOSH at $\bar{x}$)
if there are $\delta>0$ and $M>0$ such that for any point $x\in[\mathbb{B}_{\delta}(\bar{x})\cap C]\backslash\{\bar{x}\}$
and $v\in N_{C}(x)$ such that $\|v\|=1$, we have 
\begin{equation}
\left\langle v,x-\bar{x}\right\rangle \leq M\|\bar{x}-x\|^{2}.\label{eq:SOSH-1}
\end{equation}

\end{defn}
It is clear how \eqref{eq:little-SOSH} compares with \eqref{eq:SOSH-1}.
The next two results show that SOSH is prevalent in applications.
\begin{prop}
\label{prop:C2-SOSH}(Smoothness implies SOSH) Suppose function $f:\mathbb{R}^{n}\to\mathbb{R}$
is $\mathcal{C}^{2}$ at $\bar{x}$. Then the set $C=\{x\mid f(x)\leq0\}$
is SOSH at $\bar{x}$.\end{prop}
\begin{proof}
Consider $\bar{x},x\in C$. In order for the problem to be meaningful,
we shall only consider the case where $f(\bar{x})=0$. We also assume
that $f(x)=0$ so that $C$ has a tangent hyperplane at $x$. An easy
calculation gives $N_{C}(\bar{x})=\mathbb{R}_{+}\{\nabla f(\bar{x})\}$
and $N_{C}(x)=\mathbb{R}_{+}\{\nabla f(x)\}$.

Without loss of generality, let $\bar{x}=0$. We have 
\begin{eqnarray*}
 &  & f(x)=f(0)+\nabla f(0)x+\frac{1}{2}x^{T}\nabla^{2}f(0)x+o(\|x\|^{2}).\\
 & \Rightarrow & f(0)x+\frac{1}{2}x^{T}\nabla^{2}f(0)x=o(\|x\|^{2}).
\end{eqnarray*}
Since $f(x)=f(0)=0$ and $[\nabla f(0)-\nabla f(x)]x=x^{T}\nabla^{2}f(0)x+o(\|x\|^{2})$,
we have 
\[
-\nabla f(x)(x)=[\nabla f(0)-\nabla f(x)]x+\frac{1}{2}x^{T}\nabla^{2}f(0)x+o(\|x\|^{2})=O(\|x\|^{2}).
\]
 Therefore, we are done. \end{proof}
\begin{prop}
\label{prop:SOSH-intersect}(SOSH under intersection) Suppose $K_{l}\subset\mathbb{R}^{n}$
are closed sets that are SOSH at $\bar{x}$ for $l\in\{1,\dots,m\}$.
Let $K:=\cap_{l=1}^{m}K_{l}$, and suppose that $\{K_{l}\}_{l=1}^{m}$
satisfy the linear regular intersection property at $\bar{x}$. Then
$K$ is SOSH at $\bar{x}$.\end{prop}
\begin{proof}
Since each $K_{l}$ is SOSH at $\bar{x}$, we can find $\delta>0$
and $M>0$ such that for all $l\in\{1,\dots,m\}$ and $x\in K_{l}\cap\mathbb{B}_{\delta}(\bar{x})$
and $v\in N_{K_{l}}(x)$, we have 
\[
\left\langle v,x-\bar{x}\right\rangle \leq M\|v\|\|\bar{x}-x\|^{2}.
\]

\textbf{\uline{Claim 1: We can reduce $\delta>0$ if necessary
so that}}
\begin{eqnarray}
 &  & \sum_{l=1}^{m}v_{l}=0,\, v_{l}\in N_{K_{l}}(x)\mbox{ and }x\in K\cap\mathbb{B}_{\delta}(\bar{x})\label{eq:CQ-all-x}\\
 &  & \qquad\qquad\mbox{ implies }v_{l}=0\mbox{ for all }l\in\{1,\dots,m\}.\nonumber 
\end{eqnarray}
Suppose otherwise. Then we can find $\{x_{i}\}_{i=1}^{\infty}\in K$
such that $\lim x_{i}=\bar{x}$ and for all $i>0$, there exists $v_{l,i}\in N_{K_{l}}(x_{i})$
such that $\sum_{l=1}^{m}v_{l,i}=0$ but not all $v_{l,i}=0$. We
can normalize so that $\|v_{l,i}\|\leq1$, and for each $i$, $\max_{l}\|v_{l,i}\|=1$.
By taking a subsequence if necessary, we can assume that $\lim v_{l,i}$,
say $\bar{v}_{l}$, exists for all $l$. Not all $\bar{v}_{l}$ can
be zero, but $\sum_{l=1}^{m}\bar{v}_{l}=0$. The outer semicontinuity
of the normal cone mapping implies that $\bar{v}_{l}\in N_{K_{l}}(\bar{x})$.
This contradicts the linear regular intersection property, which ends
the proof of Claim 1.

\textbf{\uline{Claim 2: There exists a constant $M^{\prime}$ such
that whenever $x\in\mathbb{B}_{\delta}(\bar{x})\cap K$, $v_{l}\in N_{K_{l}}(x)$
and $v=\sum_{l=1}^{m}v_{l}$, then $\max\|v_{l}\|\leq M^{\prime}\|v\|$. }}

Suppose otherwise. Then for each $i$, there exists $x_{i}\in\mathbb{B}_{\delta}(\bar{x})\cap K$
and $\tilde{v}_{l,i}\in N_{K_{l}}(x_{i})$ such that $\tilde{v}_{i}=\sum_{l=1}^{m}\tilde{v}_{l,i}$,
$\|\tilde{v}_{i}\|\leq\frac{1}{i}$, and $\max_{l}\|\tilde{v}_{l,i}\|=1$
for all $i$. As we take limits to infinity, this would imply that
\eqref{eq:CQ-all-x} is violated, a contradiction. This ends the proof
of Claim 2.

Since \eqref{eq:CQ-all-x} is satisfied, this means that $N_{K}(x)=\sum_{l=1}^{m}N_{K_{l}}(x)$
for all $x\in\mathbb{B}_{\delta}(\bar{x})\cap K$ by the intersection
rule for normal cones in \cite[Theorem 6.42]{RW98}. Then each $v\in N_{K}(x)$
can be written as a sum of elements in $N_{K_{l}}(x)$, say $v=\sum_{l=1}^{m}v_{l}$,
where $v_{l}\in N_{K_{l}}(x)$, and $\max\|v_{l}\|\leq M^{\prime}\|v\|$.
Then 
\begin{eqnarray*}
\langle v,x-\bar{x}\rangle & = & \sum_{l=1}^{m}\langle v_{l},x-\bar{x}\rangle\\
 & \leq & M\|\bar{x}-x\|^{2}\sum_{l=1}^{m}\|v_{l}\|\quad\leq\quad M\|\bar{x}-x\|^{2}mM^{\prime}\|v\|.
\end{eqnarray*}
Thus we are done.
\end{proof}
We now make a connection to the Newton method. Consider the mass
projection algorithm.
\begin{thm}
(Connection to Newton method) Consider Algorithm \ref{alg:basic-alg}
for the case when $S_{1}=\{1,\dots,m\}$ and $S_{j}=\emptyset$ for
all $j\in\{2,\dots,m\}$ at all iterations $i$, and $\tilde{\ensuremath{S}}_{i}^{j}=\{j\}\times S_{j}$
for all $j\in\{1,\dots,m\}$. See Remark \ref{rem:mass-proj}. Let
$x^{*}\in K:=\cap_{l=1}^{m}K_{l}$. Suppose the following hold
\begin{enumerate}
\item Each set $K_{l}$ is super-regular. 
\item For each $l\in\{1,\dots,m\}$, $K_{l}$ is either a manifold, or $N_{K_{l}}(x)$
contains at most one point of norm 1 for all $x\in K_{l}$ near $x^{*}$. 
\item The sets $\{K_{l}\}_{l=1}^{m}$ has linearly regular intersection
at $x^{*}$. 
\end{enumerate}
Then provided $x_{0}$ is close enough to $x^{*}$, the convergence
of the iterates $\{x_{i}\}$ to some $\bar{x}\in K$ is superlinear.
Furthermore, the convergence is quadratic if all the sets $K_{i}$
satisfy the SOSH property.\end{thm}
\begin{proof}
By Theorem \ref{thm:loc-lin-conv}, the convergence of the iterates
$\{x_{i}\}$ to $\bar{x}$ is assured. What remains is to prove that
the convergence is actually superlinear, or quadratic under the additional
assumption. Without loss of generality, let $\bar{x}=0$. We first
prove the superlinear convergence. The proof in Theorem \ref{thm:loc-lin-conv}
assures that there is some $\beta\geq1$ such that $d(x_{i},K)\leq\beta\max_{l}d(x_{i},K_{l})$
for all iterates $x_{i}$.

Let $x_{i}$ be an iterate. Recall that $x_{i,1,j}=P_{K_{j}}(x_{i})$.
The projection of $x_{i}$ onto the polyhedron gives $x_{i+1}$. Let
$v_{j}^{+}$ be the unit normal in $N_{K_{j}}(x_{i+1,1,j})$ in the
direction of $x_{i+1}-x_{i+1,1,j}$, and let $v_{j}^{\circ}$ be the
unit normal in $N_{K_{j}}(x_{i,1,j})$ that is close to $v_{j}^{+}$. 

The proof of Theorem \ref{thm:loc-lin-conv} uses Lemma \ref{lem:lin-conv-backbone}.
Hence there are constants $c$ and $\rho\in(0,1)$ such that $\|x_{i}\|\leq\frac{c}{1-\rho}d(x_{i},K)$
for all $i$. By local metric inequality, let the index $j$ be such
that $d(x_{i+1},K)\leq\beta d(x_{i+1},K_{j})$. We let $\kappa=\frac{c\beta}{(1-\rho)}$.
Then 
\begin{eqnarray}
\langle v_{j}^{+},x_{i+1}-x_{i+1,1,j}\rangle & = & \|x_{i+1}-x_{i+1,1,j}\|\label{eq:first-chain}\\
 & = & d(x_{i+1},K_{j})\geq\frac{1}{\beta}d(x_{i+1},K)\geq\frac{1}{\kappa}\|x_{i+1}\|.\nonumber 
\end{eqnarray}
Consider the neighborhood $U$ such that if $x\in U$ and $v\in N_{K_{j}}(x)\backslash\{0\}$,
then ${\|\frac{v}{\|v\|}-\frac{\bar{v}}{\|\bar{v}\|}\|}\leq\frac{1}{4\kappa}$
for some $\bar{v}\in N_{K_{j}}(\bar{x})\backslash\{0\}$. If $i$
is large enough, then $x_{i}\in U$ and $x_{i,1,j}\in U$ for all
$j\in\{1,\dots,m\}$, which leads to 
\begin{equation}
\|v_{j}^{\circ}-v_{j}^{+}\|\leq\|v_{j}^{\circ}-\bar{v}_{j}\|+\|v_{j}^{+}-\bar{v}_{j}\|\leq\frac{1}{2\kappa},\label{eq:second-chain}
\end{equation}
where $\bar{v}_{j}$ is the appropriate unit vector in $N_{K_{j}}(\bar{x})$.
For any $\delta>0$, we can reduce the neighborhood $U$ if necessary
so that by super-regularity, 
\begin{equation}
\langle v_{j}^{+},0-x_{i+1,1,j}\rangle\leq\delta\|x_{i+1,1,j}\|.\label{eq:super-reg-conseq}
\end{equation}

\textbf{Claim: }$\|x_{i+1,1,j}\|\leq\frac{1}{\sqrt{1-\delta^{2}}}\|x_{i+1}\|$. 

We know that $x_{i+1}=x_{i+1,1,j}+tv_{j}^{+}$, where $t=\|x_{i+1}-x_{i+1,1,j}\|>0$.
By super-regularity, we have $\cos^{-1}\delta\leq\angle x_{i+1}x_{i+1,1,j}0$.
Note that $\sqrt{1-\delta^{2}}=\sin\cos^{-1}\delta$. Some simple
trigonometry ends the proof of the claim.

Choose $\delta$ small enough so that $\frac{\delta}{\sqrt{1-\delta^{2}}}\leq\frac{1}{4\kappa}$.
From \eqref{eq:super-reg-conseq}, we have 
\begin{equation}
\langle v_{j}^{+},0-x_{i+1,1,j}\rangle\leq\delta\|x_{i+1,1,j}\|\leq\frac{\delta}{\sqrt{1-\delta^{2}}}\|x_{i+1}\|\leq\frac{1}{4\kappa}\|x_{i+1}\|.\label{eq:third-chain}
\end{equation}
Then combining \eqref{eq:first-chain}, \eqref{eq:third-chain} and
\eqref{eq:second-chain}, we get 
\begin{eqnarray}
\langle v_{j}^{\circ},x_{i+1}\rangle & = & \langle v_{j}^{+},x_{i+1}-x_{i+1,1,j}\rangle+\langle v_{j}^{+},x_{i+1,1,j}\rangle+\langle v_{j}^{\circ}-v_{j}^{+},x_{i+1}\rangle\label{eq:suplin-set-1}\\
 & \geq & \frac{1}{\kappa}\|x_{i+1}\|-\frac{1}{4\kappa}\|x_{i+1}\|-\frac{1}{2\kappa}\|x_{i+1}\|=\frac{1}{4\kappa}\|x_{i+1}\|.\nonumber 
\end{eqnarray}
Choose any $\epsilon>0$. Theorem \ref{thm:radiality} implies that
$\langle v_{j}^{\circ},x_{i,1,j}\rangle\leq\epsilon\|x_{i,1,j}\|$
for all $i$ large enough. We have the following set of inequalities.
\begin{equation}
\langle v_{j}^{\circ},x_{i+1}\rangle\leq\langle v_{j}^{\circ},x_{i,1,j}\rangle\leq\epsilon\|x_{i,1,j}\|\leq\frac{\epsilon}{\sqrt{1-\delta^{2}}}\|x_{i}\|.\label{eq:suplin-set-2}
\end{equation}
(The first inequality comes from the fact that $x_{i+1}$ has to lie
in the halfspaces constructed by the previous projection. If $K_{j}$
is a manifold, then the first inequality is in fact an equation. The
last inequality is from the highlighted claim above.) Combining \eqref{eq:suplin-set-1}
and \eqref{eq:suplin-set-2} gives $\|x_{i+1}\|\leq\frac{4\kappa\epsilon}{\sqrt{1-\delta^{2}}}\|x_{i}\|$,
which is what we need.

In the case where $K_{j}$ has the SOSH property near $\bar{x}$,
\eqref{eq:suplin-set-2} can be changed to give $\langle v_{j}^{\circ},x_{i+1}\rangle\leq\frac{M}{1-\delta^{2}}\|x_{i}\|^{2}$
for some constant $M$, which gives $\|x_{i+1}\|\leq\frac{4\kappa M}{1-\delta^{2}}\|x_{i}\|^{2}$.
This completes the proof.
\end{proof}

\section{\label{sec:fast-alg}An algorithm with arbitrary fast linear convergence }

In this section, we show the arbitrary fast linear convergence of
Algorithm \ref{alg:Mass-proj-alg} for the nonconvex SIP when the
sets are super-regular. Motivated by the fast convergent algorithm
in \cite{cut_Pang12}, Algorithm \ref{alg:Mass-proj-alg} collects
old halfspaces from previous projections to try to accelerate the
convergence in later iterations. 

We now present an algorithm that can achieve arbitrarily fast linear
convergence.
\begin{algorithm}
\label{alg:Mass-proj-alg} (Local super-regular SHQP) Let $K_{l}$
be (not necessarily convex) closed sets in $\mathbb{R}^{n}$ for $l\in\{1,\dots,m\}$.
From a starting point $x_{0}\in\mathbb{R}^{n}$, this algorithm finds
a point in the intersection $K:=\cap_{l=1}^{m}K_{l}$.

\textbf{Step 0}: Set $i=1$, and let $\bar{p}$ be some positive integer.

\textbf{Step 1:} Choose $\bar{j}_{i}\in\arg\max_{j}\{d(x_{i-1},K_{j})\}$.
(i.e., we take only an index which give the largest distance.)

\textbf{Step 2:} Choose some $\tau_{i}\in[0,1)$. Define $x_{i}^{(\bar{j}_{i})}\in\mathbb{R}^{n}$,
$a_{i}^{(\bar{j}_{i})}\in\mathbb{R}^{n}$ and $b_{i}^{(\bar{j}_{i})}\in\mathbb{R}$
by\begin{subequations} \label{eq:alg-x-a-b} 
\begin{eqnarray}
x_{i}^{(\bar{j}_{i})} & \in & P_{K_{j}}(x_{i-1}),\label{eq:alg-x}\\
a_{i}^{(\bar{j}_{i})} & = & x_{i-1}-x_{i}^{(\bar{j}_{i})},\label{eq:alg-a}\\
\mbox{and }b_{i}^{(\bar{j}_{i})} & = & \langle a_{i}^{(\bar{j}_{i})},x_{i}^{(\bar{j}_{i})}\rangle+\tau_{i}\langle a_{i}^{(\bar{j}_{i})},x_{i-1}-x_{i}^{(\bar{j}_{i})}\rangle\label{eq:alg-b}\\
 & = & \langle a_{i}^{(\bar{j}_{i})},(1-\tau_{i})x_{i}^{(\bar{j}_{i})}+\tau_{i}x_{i-1}\rangle.\nonumber 
\end{eqnarray}
\end{subequations}Let $x_{i}=P_{\tilde{F}_{i}}(x_{i-1})$, where
the set $\tilde{F}_{i}\subset\mathbb{R}^{n}$ is defined by
\begin{equation}
\tilde{F}_{i}:=\big\{ x\mid\langle a_{l}^{(\bar{j}_{l})},x\rangle\leq b_{l}^{(\bar{j}_{l})}\mbox{ for }\max(1,i-\bar{p})\leq l\leq i\big\}.\label{eq:def-F}
\end{equation}

\textbf{Step 3: }Set $i\leftarrow i+1$, and go back to step 1. 
\end{algorithm}
There are some differences between Algorithm \ref{alg:Mass-proj-alg}
and that of \cite[Algorithm 5.1]{cut_Pang12}. Firstly, in step 1,
we take only one index $j$ in $\{1,\dots,m\}$ that gives the largest
distance $d(x_{i-1},K_{j})$. Secondly, the term $\tau_{i}\langle a_{i}^{(\bar{j}_{i})},x_{i-1}-x_{i}^{(\bar{j}_{i})}\rangle$
is added in \eqref{eq:alg-b} to account for the nonconvexity of the
set $K_{\bar{j}_{i}}$. 

The parameter $\tau_{i}$ in Algorithm \ref{alg:Mass-proj-alg} requires
tuning to achieve fast convergence. This tuning may not be easy to
perform. 
\begin{lem}
\label{lem:conv-alg}(Convergence of Algorithm \ref{alg:Mass-proj-alg})
Suppose that in Algorithm \ref{alg:Mass-proj-alg}, the sets $K_{l}$
are all super-regular at a point $x^{*}\in K=\cap_{l=1}^{m}K_{l}$
for all $l\in\{1,\dots,m\}$, and the local metric inequality holds,
i.e., there is a $\beta>0$ and a neighborhood $V_{1}$ of $x^{*}$
such that 
\begin{equation}
d(x,\cap_{l=1}^{m}K_{l})\leq\beta\max_{1\leq l\leq m}d(x,K_{l})\mbox{ for all }x\in V_{1}.\label{eq:due-to-beta}
\end{equation}
Then for any $\tau\in(0,1)$, we can find a neighborhood $U$ of $x^{*}$
such that 
\begin{itemize}
\item For any $x_{0}\in U$, Algorithm \ref{alg:Mass-proj-alg} with $\tau_{i}=\tau$
for all $i$ generates a sequence $\{x_{i}\}$ that converges to some
$\bar{x}\in V_{1}$ so that 
\begin{eqnarray}
 &  & \|x_{i+1}-\bar{x}\|\leq\|x_{i}-\bar{x}\|\mbox{ for all }i\geq0,\label{eq:fej-mon}\\
 & \mbox{and } & \|x_{i}-\bar{x}\|\leq L\max_{l\in\{1,\dots,m\}}d(x_{i},K_{l}),\label{eq:L-bdd}
\end{eqnarray}
where 
\begin{equation}
\rho:=\frac{\sqrt{\beta^{2}-(1-\tau)^{2}}}{\beta}\mbox{ and }L:=\frac{\beta}{1-\rho}.\label{eq:r-and-L}
\end{equation}

\end{itemize}
\end{lem}
\begin{proof}
By the super-regularity of the sets $K_{l}$, for any $\delta>0$,
there exists a neighborhood $V_{2}$ of $x^{*}$ such that for any
$l\in\{1,\dots,m\}$, we have 
\begin{equation}
\langle z-y,v\rangle\leq\delta\|z-y\|\|v\|\mbox{ for all }z,y\in K_{l}\cap V_{2},v\in N_{K_{l}}(y).\label{eq:angle-small}
\end{equation}
We choose $\delta\geq0$ to be small enough so that $\delta\leq\frac{\tau(1-\rho)}{2\beta^{2}}$.

\textbf{Claim:} If $x_{i-1}$ are such that $\mathbb{B}(x_{i-1},\frac{1}{1-\rho}d(x_{i-1},K))\subset V_{1}\cap V_{2}$,
then $K\cap\mathbb{B}(x_{i-1},\frac{1}{1-\rho}d(x_{i-1},K))\subset H_{i}$,
where the halfspace $H_{i}:=\{x:\langle a_{i}^{(\bar{j}_{i})},x\rangle\leq b_{i}^{(\bar{j}_{i})}\}$
is defined by \eqref{eq:alg-x-a-b}. 

\textbf{Proof of Claim: }Suppose $x^{\prime}\in K\cap\mathbb{B}(x_{i-1},\frac{1}{1-\rho}d(x_{i-1},K))$.
Since $K\cap V_{2}$, we have
\[
\langle x_{i-1}-x_{i}^{(\bar{j}_{i})},x^{\prime}-x_{i}^{(\bar{j}_{i})}\rangle\leq\delta\|x_{i-1}-x_{i}^{(\bar{j}_{i})}\|\|x^{\prime}-x_{i}^{(\bar{j}_{i})}\|,
\]
where $x_{i}^{(\bar{j}_{i})}$ is the point in $P_{K_{j_{i}}}(x_{i-1})\subset K_{j_{i}}$
in \eqref{eq:alg-x-a-b}. Also, $x^{\prime}$ was assumed to lie in
$\mathbb{B}(x_{i-1},\frac{1}{1-\rho}d(x_{i-1},K))$. Note that $\|x_{i-1}-x_{i}^{(\bar{j}_{i})}\|\leq d(x_{i-1},K)$.
So we have 
\[
\|x^{\prime}-x_{i}^{(\bar{j}_{i})}\|\leq\|x^{\prime}-x_{i-1}\|+\|x_{i-1}-x_{i}^{(\bar{j}_{i})}\|\leq\left(\frac{1}{1-\rho}+1\right)d(x_{i-1},K).
\]
Note that $\frac{1}{1-\rho}+1\leq\frac{2}{1-\rho}$. From the above
inequality, we have 
\[
\langle x_{i-1}-x_{i}^{(\bar{j}_{i})},x^{\prime}-x_{i}^{(\bar{j}_{i})}\rangle\leq\delta\|x_{i-1}-x_{i}^{(\bar{j}_{i})}\|\|x^{\prime}-x_{i}^{(\bar{j}_{i})}\|\leq\frac{2\delta}{1-\rho}d(x_{i-1},K)^{2}.
\]
Recall that $\delta\leq\frac{\tau(1-\rho)}{2\beta^{2}}$. Local metric
inequality gives $\|x_{i-1}-x_{i}^{(\bar{j}_{i})}\|\geq\frac{1}{\beta}d(x_{i},K)$,
so
\[
\langle x_{i-1}-x_{i}^{(\bar{j}_{i})},x^{\prime}-x_{i}^{(\bar{j}_{i})}\rangle\leq\frac{\tau}{\beta^{2}}d(x_{i-1},K)^{2}\leq\tau\|x_{i-1}-x_{i}^{(\bar{j}_{i})}\|^{2}.
\]
The above inequality is precisely $\langle a_{i}^{(\bar{j}_{i})},x^{\prime}\rangle\leq b_{i}^{(\bar{j}_{i})}$,
so $x^{\prime}\in H_{i}$. This ends the proof of the claim.

Suppose $\mathbb{B}(x_{0},\frac{1}{1-\rho}d(x_{0},K))\subset V_{1}\cap V_{2}$.
If the conditions of Lemma \ref{lem:lin-conv-backbone} are satisfied,
then we have convergence to some $\bar{x}$. 

We try to prove that $d(x_{i+1},K)\leq\rho d(x_{i},K)$. Recall that
$x_{i+1}=P_{\tilde{F}_{i+1}}(x_{i})$. By making use of the claim
above, the previous halfspaces generated all contain $K\cap\mathbb{B}(x_{i},\frac{1}{1-\rho}d(x_{i},K))$,
so $\tilde{F}_{i+1}$ is a polyhedron that contains $K\cap\mathbb{B}(x_{i},\frac{1}{1-\rho}d(x_{i},K))$.
It is clear that $K\cap\mathbb{B}(x_{i},\frac{1}{1-\rho}d(x_{i},K))\neq\emptyset$,
so $\tilde{F}_{i+1}$ is nonempty. It is obvious that $d(x_{i},\tilde{F}_{i+1})\leq d(x_{i},K)$,
so $\|x_{i}-x_{i+1}\|\leq(1-\tau)d(x_{i},K)$. The distance $d(x_{i},H_{i+1})$
is at least $\frac{1}{\beta}d(x_{i},K)$, so $\|x_{i}-x_{i+1}\|\geq\frac{1}{\beta}d(x_{i},K)$.
We then have 
\begin{eqnarray*}
d(x_{i+1},K)^{2} & \leq & d(x_{i},K)^{2}-\|x_{i}-x_{i+1}\|^{2}\\
 & \leq & d(x_{i},K)^{2}-\frac{(1-\tau)^{2}}{\beta^{2}}d(x_{i},K)^{2}\\
 & = & \frac{\sqrt{\beta^{2}-(1-\tau)^{2}}}{\beta}d(x_{i},K)^{2}.
\end{eqnarray*}
We can now apply Lemma \ref{lem:lin-conv-backbone}. The conclusion
\eqref{eq:fej-mon} comes from the fact that $\{x_{i}\}$, by construction,
is obtained by projection onto convex sets that contain $\bar{x}$
and the theory of Fej\'{e}r monotonicity. The conclusion \eqref{eq:L-bdd}
is straightforward from Lemma \ref{lem:lin-conv-backbone}(a) and
local metric inequality.
\end{proof}
We now prove the theorem on the arbitrary fast multiple-term linear
convergence of Algorithm \ref{alg:Mass-proj-alg}.
\begin{thm}
\label{thm:arb-lin-conv}(Arbitrary fast linear convergence) Consider
the setting of Theorem \ref{lem:conv-alg}. If $\bar{p}$ in Algorithm
\ref{alg:Mass-proj-alg} is finite and sufficiently large, then for
all $\tau\in(0,0.5)$ (independent of $\bar{p}$) we can find a neighborhood
$U$ of $x^{*}$ such that if $x_{0}\in U$, then the iterates of
Algorithm \ref{alg:Mass-proj-alg} with $\tau_{i}\equiv\tau$ converge
to some $\bar{x}\in K$. Moreover, 
\begin{equation}
\limsup_{i\to\infty}\frac{\|x_{i+\bar{p}}-\bar{x}\|}{\|x_{i}-\bar{x}\|}\leq8\bar{L}\tau,\label{eq:arb-lin-concl}
\end{equation}
where $\bar{L}=\frac{\beta}{1-\bar{\rho}}$ and $\bar{\rho}=\frac{\sqrt{\beta^{2}-1/4}}{\beta}$.\end{thm}
\begin{proof}
The basic strategy is to prove the inequalities \eqref{eq:money-1}
and \eqref{eq:money-2} like in \cite[Theorem 5.12]{cut_Pang12},
with a bit more attention put into handling the nonconvexity.

By Lemma \ref{lem:conv-alg}, the convergence of the iterates $\{x_{i}\}$
to some $\bar{x}\in K$ is assured. Without loss of generality, suppose
that $\bar{x}=0$. Let $v_{i}^{*}:=\frac{x_{i}-x_{i+1}^{(\bar{j}_{i+1})}}{\|x_{i}-x_{i+1}^{(\bar{j}_{i+1})}\|}$,
where $x_{i+1}^{(\bar{j}_{i+1})}$ is defined through \eqref{eq:alg-x}. 

The sphere $S^{n-1}:=\{w\in\mathbb{R}^{n}:\|w\|=1\}$ is compact.
Suppose $\bar{p}$ is such that we can cover $S^{n-1}$ with $\bar{p}$
balls of radius $\frac{1}{4\bar{L}}$. By the pigeonhole principle,
we can find $j$ and $k$ such that $i\leq j<k\leq i+\bar{p}$ and
$v_{j}^{*}$ and $v_{k}^{*}$ belong to the same ball of radius $\frac{1}{4\bar{L}}$
covering $S^{n-1}$. We thus have $\|v_{j}^{*}-v_{k}^{*}\|\leq\frac{1}{2\bar{L}}$.
(The key in choosing $\bar{p}$ is to obtain the last inequality.)

We shall prove that if $i$ is large enough, we have the two inequalities
\begin{eqnarray}
 &  & \langle v_{j}^{*},x_{k}\rangle\leq2\tau\|x_{j}\|\label{eq:money-1}\\
 & \mbox{ and } & \frac{1}{4\bar{L}}\|x_{k}\|\leq\langle v_{j}^{*},x_{k}\rangle.\label{eq:money-2}
\end{eqnarray}
In view of the Fej\'{e}r monotonicity condition \eqref{eq:fej-mon},
these two inequalities give $\|x_{i+\bar{p}}\|\leq\|x_{k}\|\leq8\bar{L}\tau\|x_{j}\|\leq8\bar{L}\tau\|x_{i}\|$,
which gives the conclusion we seek.

We first prove \eqref{eq:money-1}. Since $x_{k}$ lies in $\tilde{F}_{k}$,
it lies in the halfspace with normal $v_{j}^{*}$ passing through
$(1-\tau)x_{j+1}^{(\bar{j}_{j+1})}+\tau x_{j}$. (Recall that $x_{j+1}^{(\bar{j}_{j+1})}$
was defined in \eqref{eq:alg-x}, and lies in $P_{K_{\bar{j}_{j+1}}}(x_{j})$.)
This gives us 
\begin{eqnarray}
\langle v_{j}^{*},x_{k}\rangle & \leq & \langle v_{j}^{*},(1-\tau)x_{j+1}^{(\bar{j}_{j+1})}+\tau x_{j}\rangle\label{eq:money-1-a}\\
 & = & (1-\tau)\langle v_{j}^{*}-\bar{v},x_{j+1}^{(\bar{j}_{j+1})}\rangle+(1-\tau)\langle\bar{v},x_{j+1}^{(\bar{j}_{j+1})}\rangle+\tau\langle v_{j}^{*},x_{j}\rangle,\nonumber 
\end{eqnarray}
where $\bar{v}$ is some vector with norm 1 in $N_{K_{\bar{j}_{i+1}}}(\bar{x})$.
Since $\lim_{i\to\infty}x_{i}=\bar{x}$, we can assume that $\{x_{i}\}$
is sufficiently close to $\bar{x}$ so that:
\begin{enumerate}
\item the vector $\bar{v}$, by the outer semicontinuity of the normal cone
mapping $x\mapsto N_{K_{\bar{j}_{j+1}}}(x)$, can be chosen to be
such that $\|v_{j}^{*}-\bar{v}\|\leq\frac{\tau}{3}$, and 
\item by the super-regularity of $K_{\bar{j}_{j+1}}$ at $\bar{x}$, we
have $\langle\bar{v},x_{j+1}^{(\bar{j}_{i+1})}\rangle\leq\frac{\tau}{3}\|x_{j+1}^{(\bar{j}_{i+1})}\|$.
\end{enumerate}
Note that $(1-\tau)x_{j+1}^{(\bar{j}_{i+1})}+\tau x_{j}$ is the projection
of $x_{j}$ onto one of the halfspaces defining $\tilde{F}_{j+1}$
and that $\tau<\frac{1}{2}$. From the principle in Proposition \ref{prop:fejer-principle},
we have $\|x_{j+1}^{(\bar{j}_{i+1})}\|\leq\|x_{j}\|$. Since $\|v_{j}^{*}\|=1$,
we have $\langle v_{j}^{*},x_{j}\rangle\leq\|x_{j}\|$. Continuing
the arithmetic in \eqref{eq:money-1-a}, we have 
\begin{eqnarray*}
\langle v_{j}^{*},x_{k}\rangle & \leq & (1-\tau)\langle v_{j}^{*}-\bar{v},x_{j+1}^{(\bar{j}_{j+1})}\rangle+(1-\tau)\langle\bar{v},x_{j+1}^{(\bar{j}_{j+1})}\rangle+\tau\langle v_{j}^{*},x_{j}\rangle\\
 & \leq & (1-\tau)\Big(\frac{\tau}{3}+\frac{\tau}{3}\Big)\|x_{j+1}^{(\bar{j}_{j+1})}\|+\tau\|x_{j}\|\\
 & \leq & \frac{2\tau}{3}\|x_{j}\|+\tau\|x_{j}\|<2\tau\|x_{j}\|.
\end{eqnarray*}
This ends the proof of \eqref{eq:money-1}. Next, we prove \eqref{eq:money-2}.
Recall that $x_{k+1}^{(\bar{j}_{k+1})}\in P_{K_{\bar{j}_{k+1}}}(x_{k})$
was defined in \eqref{eq:alg-x}. Note that provided $\tau<\frac{1}{2}$,
the $\rho=\frac{\sqrt{\beta^{2}-(1-\tau)^{2}}}{\beta}$ in \eqref{eq:r-and-L}
is less than $\bar{\rho}=\frac{\sqrt{\beta^{2}-1/4}}{\beta}$. Hence
the $L=\frac{\beta}{1-\rho}$ in \eqref{eq:r-and-L} is less than
$\bar{L}=\frac{\beta}{1-\bar{\rho}}$. By using the definition of
$v_{k}^{*}=\frac{x_{k}-x_{k+1}^{(\bar{j}_{k+1})}}{\|x_{k}-x_{k+1}^{(\bar{j}_{k+1})}\|}$
and \eqref{eq:L-bdd}, we have 
\begin{equation}
\langle v_{k}^{*},x_{k}-x_{k+1}^{(\bar{j}_{k+1})}\rangle=\|x_{k}-x_{k+1}^{(\bar{j}_{k+1})}\|=d(x_{k},K_{\bar{j}_{k+1}})\geq\frac{1}{L}\|x_{k}\|\geq\frac{1}{\bar{L}}\|x_{k}\|.\label{eq:money-2a}
\end{equation}
By the super-regularity of $K_{\bar{j}_{k+1}}$ at $\bar{x}$ and
the fact that $\bar{x}=\lim_{i\to\infty}x_{i}$, we can assume that
$x_{k}$ is close enough to $\bar{x}$ so that 
\begin{equation}
\langle v_{k}^{*},0-x_{k+1}^{(\bar{j}_{k+1})}\rangle\leq\frac{1}{4\bar{L}}\|x_{k+1}^{(\bar{j}_{k+1})}\|\leq\frac{1}{4\bar{L}}\|x_{k}\|.\label{eq:money-2b}
\end{equation}
(Note that the inequality on the right follows from the same proof
of $\|x_{j+1}^{(\bar{j}_{j+1})}\|\leq\|x_{j}\|$.) Combining \eqref{eq:money-2a}
and \eqref{eq:money-2b} as well as $\|v_{j}^{*}-v_{k}^{*}\|\leq\frac{1}{2\bar{L}}$
gives us 
\begin{eqnarray*}
\langle v_{j}^{*},x_{k}\rangle & = & \langle v_{k}^{*},x_{k}-x_{k+1}^{(\bar{j}_{k+1})}\rangle+\langle v_{k}^{*},x_{k+1}^{(\bar{j}_{k+1})}\rangle+\langle v_{j}^{*}-v_{k}^{*},x_{k}\rangle\\
 & \geq & \frac{1}{\bar{L}}\|x_{k}\|-\frac{1}{4\bar{L}}\|x_{k}\|-\frac{1}{2\bar{L}}\|x_{k}\|=\frac{1}{4\bar{L}}\|x_{k}\|.
\end{eqnarray*}
This ends the proof of \eqref{eq:money-2}, which concludes the proof
of our result.
\end{proof}
The large parameter $\bar{p}$ is an upper bound on when we can find
$v_{j}^{*}$ and $v_{k}^{*}$ such that $\|v_{j}^{*}-v_{k}^{*}\|\leq\frac{1}{2\bar{L}}$.
We hope that the upper bound needed in a practical implementation
would be much smaller than $\bar{p}$. 
\begin{rem}
(Towards superlinear convergence) The coefficient of $8$ in \eqref{eq:arb-lin-concl}
can be reduced, but this does not detract us from the point that as
$\tau\searrow0$, the right hand side of \eqref{eq:arb-lin-concl}
goes to zero. So there is a choice of parameters $\{\tau_{i}\}_{i=1}^{\infty}$
that can be chosen at each iteration of Algorithm \ref{alg:Mass-proj-alg}
so that superlinear convergence is achieved, even though there doesn't
seem to be a good way of choosing how the parameters $\tau$ go to
zero. If the parameter $\tau$ goes to zero too fast, the Fej\'{e}r
monotonicity \eqref{eq:fej-mon} of the iterates may not be maintained,
which may mean that Lemma \ref{lem:conv-alg} may not hold, i.e.,
the iterates $\{x_{i}\}_{i=1}^{\infty}$ may not converge. Contrast
this to the convex SIP in \cite{cut_Pang12}, where setting $\tau\equiv0$
gives \emph{multiple-term superlinear convergence} 
\[
\lim_{i\to\infty}\frac{\|x_{i+\bar{p}}-\bar{x}\|}{\|x_{i}-\bar{x}\|}=0
\]
 instead of multiple-term arbitrary linear convergence \eqref{eq:arb-lin-concl}.
In view of nonconvexity, the observation in Remark \ref{rem:sometimes-empty-intersect}
has to be overcome, so we believe that this arbitrary fast convergence
is difficult to improve on in general.
\begin{rem}
(Simplification in \eqref{eq:money-2a}) The inequality $d(x_{k},K_{\bar{j}_{k+1}})\geq\frac{1}{L}\|x_{k}\|$
in \eqref{eq:money-2a} follows easily from \eqref{eq:L-bdd}. But
in \cite{cut_Pang12}, some effort was spent to prove the inequality
$\limsup_{k\to\infty}\frac{1}{\|x_{k}\|}d(x_{k},K_{\bar{j}_{k+1}})\geq\frac{1}{\beta}$.
 The proof of the multiple-term superlinear convergent algorithm for
convex problems in \cite{cut_Pang12} can thus be shortened considerably.
\end{rem}
\end{rem}
If some of the sets $K_{l}$ are known to be convex sets or affine
subspaces, then this information can be taken into account by setting
the appropriate $\tau_{i}$ to zero when creating the halfspaces defined
by \eqref{eq:alg-x-a-b}.

\section{Two step SHQP}

The algorithms in this paper need not guarantee that $\{d(x_{i},K)\}$
is nonincreasing. In this section, we give an example of additional
conditions needed for the SHQP to have this property. Consider the
following algorithm. 
\begin{algorithm}
\label{alg:2-SHQP}(2-SHQP) Let $K_{1}$, $K_{2}$ be two closed sets
in $\mathbb{R}^{n}$, and $K=K_{1}\cap K_{2}$. This algorithm tries
to find a point $x\in K$ using a starting iterate $x_{0}$.

01 Set $i=0$

02 Loop

03 $\quad$Set $x_{i+1}$ to be an element in $P_{K_{1}}(x_{i})$
and $i\leftarrow i+1$.

04 $\quad$Set $x_{i+1}$ to be an element in $P_{K_{2}}(x_{i})$
and $i\leftarrow i+1$.

05 $\quad$If $\angle x_{i-2}x_{i-1}x_{i}<\pi/2$, then 

06 $\quad\quad$set $x_{i+1}=P_{\{x:\langle x-x_{i-1},x_{i-2}-x_{i-1}\rangle\leq0,\langle x-x_{i},x_{i-1}-x_{i}\rangle\leq0\}}(x_{i})$.

07 $\quad$else

08 $\quad$$\quad$set $x_{i+1}=x_{i}$

09 $\quad$end if 

10 $\quad$$i\leftarrow i+1$

11 end loop
\end{algorithm}
In line 6, $x_{i+1}$ is the projection of $x_{i}$ onto the polyhedron
formed by intersecting the last two halfspaces generated by the projection
process. See Figure \ref{fig:no-t-fig} for an illustration of the
first few iterates $x_{1}$, $x_{2}$ and $x_{3}$ formed by a single
iteration of the loop. If the ``if'' block in lines 5 to 9 is removed,
then the algorithm reduces to an alternating projection algorithm.
We now analyze the effectiveness of this ``if'' block.
\begin{prop}
\label{prop:2-SHQP}(2-SHQP) Consider Algorithm \ref{alg:2-SHQP}.
Let $\delta\in(0,1)$. Let $x^{*}\in K$ and let a neighborhood $V$
of $x^{*}$ be such that 
\begin{enumerate}
\item $\langle v,y-z\rangle\leq\delta\|v\|\|y-z\|\mbox{ for all }y,z\in K_{l}\cap V\mbox{, }l\in\{1,2\}\mbox{ and }v\in N_{K_{l}}(z),$and
\item $d(x,K)\leq\beta\max_{l\in\{1,2\}}d(x,K_{l})$ for all $x\in V$.
\end{enumerate}
Let $x_{1}$, $x_{2}$ be successive iterates of Algorithm \ref{alg:2-SHQP}.
Suppose $\mathbb{B}(x_{1},(\beta+1)\|x_{1}-x_{2}\|)\subset V$. Let
$\theta:=\angle x_{2}x_{1}x_{0}<\pi/2$. If 
\begin{equation}
\delta[\beta\cos\theta+(\beta+1)]<\frac{1}{2}\cos\theta,\label{eq:no-t-concl}
\end{equation}
then $d(x_{3},K)<d(x_{2},K)$. 
\end{prop}
Conditions (1) and (2) are consequences of the super-regularity condition
and local metric inequality condition respectively, so they will be
satisfied when close to $K$. 
\begin{proof}
Since property (2) and the fact that $x_{1}\in K_{1}$ implies that
\[
d(x_{2},K)\leq\beta\max_{l\in\{1,2\}}d(x_{2},K_{l})=\beta d(x_{2},K_{1})\leq\beta\|x_{1}-x_{2}\|,
\]
the set $\mathbb{B}(x_{2},\beta\|x_{2}-x_{1}\|)\cap K$ is not empty.
Hence 
\begin{equation}
\emptyset\neq\mathbb{B}(x_{2},\beta\|x_{2}-x_{1}\|)\subset\mathbb{B}(x_{1},(\beta+1)\|x_{2}-x_{1}\|)\subset V.\label{eq:ball-chain}
\end{equation}
Let $y$ be any point in $\mathbb{B}(x_{2},\beta\|x_{2}-x_{1}\|)\cap K$.
By property (1), we have 
\begin{equation}
\langle y-x_{2},x_{1}-x_{2}\rangle\leq\delta\|y-x_{2}\|\|x_{1}-x_{2}\|\leq\beta\delta\|x_{1}-x_{2}\|^{2}.\label{eq:first-ball-nonempty}
\end{equation}
In other words, $\mathbb{B}(x_{2},\beta\|x_{2}-x_{1}\|)\cap K\subset H_{2}$,
where $H_{2}$ is the halfspace defined by
\[
H_{2}:=\{x:\langle x-x_{2},x_{1}-x_{2}\rangle\leq\beta\delta\|x_{1}-x_{2}\|^{2}\}.
\]
Next, from \eqref{eq:ball-chain}, we can make use of the argument
similar to \eqref{eq:first-ball-nonempty} to prove that 
\[
\mathbb{B}(x_{1},(\beta+1)\|x_{2}-x_{1}\|)\cap K\subset H_{1},
\]
where $H_{1}$ is the halfspace defined by 
\[
H_{1}:=\{x:\langle x-x_{1},x_{0}-x_{1}\rangle\leq(\beta+1)\delta\|x_{0}-x_{1}\|^{2}\}.
\]
 This implies that 
\begin{equation}
\emptyset\neq\mathbb{B}(x_{2},\beta\|x_{2}-x_{1}\|)\cap K\subset H_{1}\cap H_{2}.\label{eq:poly-contain}
\end{equation}

\begin{figure}[!h]
\includegraphics[scale=0.3]{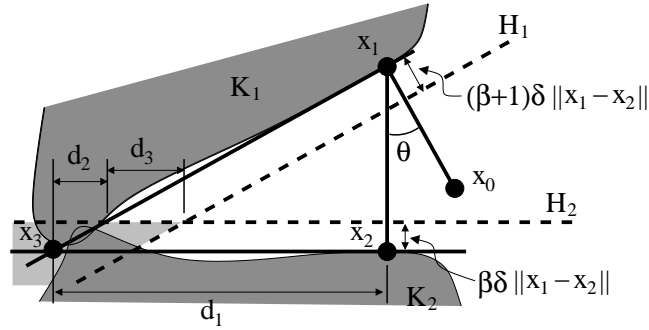}

\caption{\label{fig:no-t-fig}This figure illustrates the proof of Proposition
\ref{prop:2-SHQP}. The dotted lines show the boundaries of $H_{1}$
and $H_{2}$. }
\end{figure}

We refer to Figure \ref{fig:no-t-fig}, which shows the two dimensional
cross section containing $x_{0}$, $x_{1}$ and $x_{2}$. The point
$x_{3}$ is also shown in the figure, and is the projection of $x_{2}$
onto $H_{1}\cap H_{2}$. We now calculate the minimal value of $\langle\frac{x_{3}-x_{2}}{\|x_{3}-x_{2}\|},x\rangle$,
where $x$ ranges over $H_{1}\cap H_{2}$. This minimal value can
be seen to be $d_{1}-d_{2}-d_{3}$, where $d_{1}$, $d_{2}$ and $d_{3}$
are the distances as indicated in Figure \ref{fig:no-t-fig}. These
distances can be calculated to be 
\begin{eqnarray*}
d_{1} & = & \|x_{2}-x_{3}\|=\|x_{1}-x_{2}\|\cot\theta,\\
d_{2} & = & \beta\delta\|x_{1}-x_{2}\|\cot\theta,\\
\mbox{ and }d_{3} & = & (\beta+1)\delta\|x_{1}-x_{2}\|/\sin\theta.
\end{eqnarray*}
We can check that \eqref{eq:no-t-concl} is equivalent to $d_{2}+d_{3}<\frac{1}{2}d_{1}$.
As long as \eqref{eq:no-t-concl} holds, the region $H_{1}\cap H_{2}$
lies on the same side as $x_{3}$ of the perpendicular bisector of
the points $x_{2}$ and $x_{3}$. Hence all the points in $H_{1}\cap H_{2}$
are closer to $x_{3}$ than to $x_{2}$. Since $H_{1}\cap H_{2}$
contains all the points in $P_{K}(x_{2})$ by \eqref{eq:poly-contain},
we thus have $d(x_{3},K)<d(x_{2},K)$ as needed.
\end{proof}
Note that if $\theta<\pi/2$ is too close to $\pi/2$, then the condition
\eqref{eq:no-t-concl} can fail. In fact, if $\theta>\cos^{-1}\delta$,
one can check that condition (1) in Proposition \ref{prop:2-SHQP}
does not rule out $x_{2}$ being inside $K_{1}$, so there would be
no point calculating $x_{3}$. The supporting halfspaces as calculated
by the projection process can be too aggressive for super-regular
sets. For example, one can draw a manifold in $\mathbb{R}^{2}$ such
that the intersection the manifold and a halfspace generated by the
projection process consists of only one point. Two halfspaces of this
kind would give an empty intersection with the manifold. Therefore,
one has to relax the halfspaces. 

We remark that the procedure in \eqref{eq:first-ball-nonempty} shows
how to construct halfspaces under the super-regularity condition,
and can be augmented into Algorithm \ref{alg:Mass-proj-alg} as long
as we have a good estimate for $\delta$.

\section{\label{sec:global-strat}Global strategies}

In this section, we discuss methods for when local methods of the
nonconvex SIP are not appropriate. In Example \ref{exa:backtrack},
we show that while the theory for the convex SIP suggests that one
should not backtrack, backtracking is however suggested for the nonconvex
problem, which can lead to the Maratos effect and slows down convergence. 

The problem of finding a point in the intersection of a finite number
of closed sets $K_{l}\subset\mathbb{R}^{n}$, where $l=1,\dots,m$,
can be equivalently cast as the problem of finding a point that minimizes
$f(x)$, where $f(x)$ can be chosen as \begin{subequations}\label{eq:all-d-forms}
\begin{eqnarray}
 &  & d(x,\cap_{l=1}^{m}K_{l}),\label{eq:d-int-form}\\
 &  & \sum_{l=1}^{m}d(x,K_{l})^{2},\label{eq:d-2-norm-form}\\
 &  & \max_{l\in\{1,\dots,m\}}\{d(x,K_{l})\},\label{eq:d-max-form}
\end{eqnarray}
\end{subequations} or some other function similar to those presented
above. In the event that the intersection $\cap_{l=1}^{m}K_{l}$ is
nonempty, then any point in $K:=\cap_{l=1}^{m}K_{l}$ would be a global
minimizer of $f(\cdot)$. The function in \eqref{eq:d-int-form} is
the function of choice, but $\cap_{l=1}^{m}K_{l}$ can be only be
estimated well locally with the techniques in Section \ref{sec:local-strat}.
Instead of trying to minimize $f(\cdot)$, the problem that really
needs to be solved is the one of finding an $x$ in $\{\tilde{x}:f(\tilde{x})\leq0\}$.
This is a simpler problem which can be solved by a subgradient projection
method that is somewhat simpler than the minimization problem. A bundle
method \cite{HiriartUrrutyLamerechal93a,BonnansGilbertLemarechalSagastizabal06}
adapted for a nonconvex objective function can be used to solve the
nonconvex SIP. (See also \cite{BauschkeWangWangXu14,Pang_nonconvex_ineq}
for the principles of a finitely convergent algorithm for this setting.
This idea of finite convergence goes back to \cite{Polak_Mayne79,Mayne_Polak_Heunis81,Fukushima82,DePierroIusem88}
for the convex case and the smooth case.)

A standard procedure in optimization algorithms is the line search
procedure. A search direction is calculated, and the next solution
is obtained by a line search along this search direction. For the
nonconvex SIP, the search direction can be calculated by projecting
onto a polyhedron formed by intersecting a number of previously generated
halfspaces. There are two ways we can backtrack to obtain decrease
in some objective function (in \eqref{eq:all-d-forms} or otherwise).
Firstly, one can remove halfspaces that describe the polyhedron. It
is sensible to remove the older halfspaces since they become less
reliable. This has the effect of reducing the distance from the current
iterate to the polyhedron, so the search direction is more likely
to give decrease. The problem of projecting onto the polyhedron with
one halfspace removed can be solved effectively from the old solution
using a warmstart quadratic programming algorithm (for example, the
active set method of \cite{Goldfarb_primal}). Secondly, one can use
the usual backtracking line search. 

We note however that in the pursuit of obtaining decrease in the objective
function, we may encounter the Maratos effect (see \cite[Section 15.5]{NW06},
who in turn cited \cite{Maratos}) which slows convergence. 
\begin{example}
\label{exa:backtrack}(Backtracking slows convergence)  In this example,
we show how the SHQP strategy for a convex SIP converges quickly for
a problem, but would be slowed down by backtracking when treated as
a nonconvex SIP. Consider the sets $K_{1},K_{2}\subset\mathbb{R}^{3}$
where $K_{1}=H_{1}$ and $K_{2}=H_{2}\cap H_{3}$, where the halfspaces
$H_{1}$, $H_{2}$ and $H_{3}$ are defined by 
\begin{eqnarray*}
H_{1} & := & \{x\in\mathbb{R}^{3}:(0,1,0)x\leq0\},\\
H_{2} & := & \{x\in\mathbb{R}^{3}:(\nicefrac{1}{3},-1,0)x\leq-2\},\\
\mbox{and }H_{3} & := & \{x\in\mathbb{R}^{3}:(-1,-1,1)x\leq0\}.
\end{eqnarray*}
Let the point $x_{0}$ be $(0,1,0)$. The projection of $x_{0}$ onto
$K_{1}$ and $K_{2}$ generates the halfspaces $H_{1}$ and $H_{2}$
respectively. The projection of $x_{0}$ onto $H_{1}\cap H_{2}$ is
$x_{1}:=(-6,0,0)$. We can calculate that 
\begin{equation}
d(x_{0},K_{1})=1,\mbox{ }d(x_{0},K_{2})=\nicefrac{3}{\sqrt{10}},\mbox{ }d(x_{1},K_{1})=0\mbox{ and }d(x_{1},K_{2})=2\sqrt{3}.\label{eq:exa-vals}
\end{equation}
The projection of $x_{1}$ onto $K_{2}$ generates $H_{3}$, and once
we project $x_{1}$ onto $H_{1}\cap H_{2}\cap H_{3}$, we found a
point in $K_{1}\cap K_{2}$. If this SIP were solved as a nonconvex
SIP, the values in \eqref{eq:exa-vals} fitted into the objective
function \eqref{eq:d-2-norm-form} or \eqref{eq:d-max-form} suggests
that one has to backtrack in some manner, and this actually slows
down the convergence. (See Figure \ref{fig:exa-backtr} for an illustration.)
\end{example}
\begin{figure}[!h]
\includegraphics[scale=0.4]{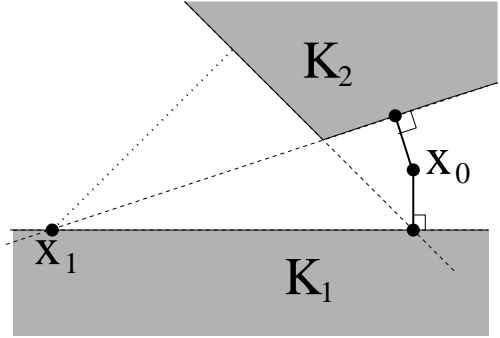}

\caption{\label{fig:exa-backtr}This figure illustrates the two dimensional
cross section in $\{x\in\mathbb{R}^{3}:x_{3}=0\}$ in the example
in Example \ref{exa:backtrack}. Note that the projection of $x_{1}$
onto $K_{2}$ lies outside this cross section.}
\end{figure}

We recall the method of averaged projections for finding a point
in $\cap_{l=1}^{m}K_{l}$, where $K_{l}\subset\mathbb{R}^{n}$ for
all $l\in\{1,\dots,m\}$, is defined by 
\begin{equation}
x_{i+1}=\frac{1}{m}\sum_{l=1}^{m}P_{K_{l}}(x_{i}).\label{eq:avg-proj}
\end{equation}
It was noticed that this formula corresponds to the method of alternating
projections between the two sets in $\mathbb{R}^{nm}$ defined by
\begin{eqnarray*}
\mathbf{D} & := & \{(x,x,\dots,x):x\in\mathbb{R}^{n}\}\\
\mbox{and }\mathbf{K} & := & K_{1}\times K_{2}\times\cdots\times K_{m}.
\end{eqnarray*}
 It is easy to see that $f(x_{i+1})\leq f(x_{i})$ if $x_{i+1}$ is
defined by \eqref{eq:avg-proj} and $f(\cdot)$ is defined by \eqref{eq:d-2-norm-form}
since $\sqrt{f(x)}$ is the distance of $(x,\cdots,x)\in\mathbf{D}$
to $\mathbf{K}$. Moreover, if $f(x_{i+1})=f(x_{i})$, then $x_{i}$
is the minimizer. 

In the SHQP strategy for nonconvex problems, we can use backtracking
to find the next iterate $x_{i}$ of the form $tP_{\tilde{F}_{i}}(x_{i-1})+(1-t)x_{i-1}$,
where $t\in(0,1]$ and $\tilde{F}_{i}$ is the polyhedron defined
by intersecting previously generated halfspaces like in Algorithm
\ref{alg:Mass-proj-alg}. We can instead find an iterate of the form
\[
tP_{\tilde{F}_{i}}(x_{i-1})+(1-t)\frac{1}{m}\sum_{l=1}^{m}P_{K_{l}}(x_{i-1}).
\]

Other heuristics for the nonconvex problem are also possible. For
example, if one is certain that the intersection is nonempty, then
one can try to avoid points in the balls $\mathbb{B}(x_{i},d(x_{i},K_{l}))$
for all $i\geq0$ and $l\in\{1,\dots,m\}$. If some of the sets are
spectral sets (i.e., the set of symmetric matrices solely described
by their eigenvalues), then the results in \cite{LewisMalick08} can
also be applied.

\section{Conclusion}

We hope our results make the case that in solving feasibility problems
involving super-regular sets, one should use the SHQP procedure as
much as possible to accelerate convergence once close enough to the
intersection. The size of the QPs to be solved can be kept to be of
a manageable size if we combine with projection methods like in Algorithm
\ref{alg:basic-alg}. \bibliographystyle{amsalpha}
\bibliography{../refs}

\end{document}